\pdfoutput=1
\documentclass[11pt]{amsart}
\usepackage[T1]{fontenc}
\usepackage{lmodern}
\usepackage{amsmath,amssymb,amsthm}
\usepackage[margin=1in]{geometry}
\usepackage[hidelinks]{hyperref}
\hypersetup{
  pdftitle={Wallis-type products with polynomial exponents and the Dirichlet beta function at negative integers},
  pdfauthor={Joshua W. E. Farrell}
}

\newtheorem{theorem}{Theorem}[section]
\newtheorem{proposition}[theorem]{Proposition}
\newtheorem{lemma}[theorem]{Lemma}

\theoremstyle{definition}
\newtheorem{definition}[theorem]{Definition}
\newtheorem{remark}[theorem]{Remark}
\newtheorem{example}[theorem]{Example}

\title[Wallis-type products and beta at negative integers]{Wallis-type products with polynomial exponents and the Dirichlet beta function at negative integers}
\author{Joshua W. E. Farrell}

\begin{document}

\begin{abstract}
We develop a methodology for designing infinite products of rational blocks whose exponents are polynomials in the index~$k$. Matching power sums of the slot constants through order~$n$ forces the Type-$N$ product with binomial exponent $\binom{n+k-2}{n-1}$ to converge to a ratio of Vign\'eras multiple gamma values~$\Gamma_{n}$; an analogue finder lifts any Type-1 evaluation to every higher type, and integer combinations of binomial exponents then realise arbitrary integer-valued polynomial exponents, yielding explicit products with exponents $k$, $k^{2}$, $k^{3}$, \ldots\ for constants such as $\pi/2$, $\sqrt{2}$, $e^{2K/\pi}$, and rational multiples of $\pi^{M!}$ (Part~I). As the main application (Part~II) we prove that for every positive integer~$n$, a finite multiple-gamma template $\mathcal{S}_{n}$ with generalised Eulerian weights $T(n,k)$ (OEIS A225118) evaluates the Duke--Imamo\u{g}lu expression $\mathcal{D}_{n}=\beta'(-n)+(\log 4)\,\beta(-n)$. For odd~$n$ this yields a convergent Wallis--Eulerian product for~$e^{\beta'(-n)}$; for even~$n$ the raw product diverges. The proof expands the template through the multiple-gamma functional equation, evaluates the quarter-integer coefficients in closed form, and identifies the resulting Eulerian--binomial sums with Duke's polynomials~$P_{n+1,\ell}$.
\end{abstract}

\keywords{Dirichlet beta function, multiple gamma function, Wallis product, Eulerian numbers}
\subjclass[2020]{Primary 11M35; Secondary 33B15, 11B68}

\maketitle

\setcounter{tocdepth}{1}
\tableofcontents

\section{Introduction}\label{sec:intro}

In 1655 Wallis discovered his celebrated product $\frac{\pi}{2}=\frac{2\cdot 2}{1\cdot 3}\cdot\frac{4\cdot 4}{3\cdot 5}\cdot\frac{6\cdot 6}{5\cdot 7}\cdots$~\cite{ARTICLE:4,ARTICLE:5}. A modern proof multiplies Weierstrass factorisations of the gamma function: whenever the first moments of two slot sequences agree, the exponential corrections cancel and the product of blocks converges to a ratio of gamma values~\cite{huylebrouck2015generalizing}. This paper is about what happens when one keeps going. Each further matching power sum buys one more degree of polynomial growth in the block exponent, with the limit an explicit ratio of Vign\'eras multiple gamma values $\Gamma_{n}$~\cite{ARTICLE:2}: second moments give Barnes $G$-function products with exponent~$k$~\cite{barnes1899gfunction,article:6}, and power sums through order~$n$ give Type-$N$ products with exponent~$\binom{n+k-2}{n-1}$. Part~I turns this hierarchy into a design method (choose slot constants, lift them to any type, combine binomial exponents into arbitrary integer-valued polynomials), and Part~II applies the method to prove an identity for the Dirichlet beta function
\[
\beta(s)=\sum_{m=0}^{\infty}\frac{(-1)^{m}}{(2m+1)^{s}}
\]
at negative integers, where closed forms at positive integers~\cite{apostol1998analytic} and special-value relations with Euler numbers~\cite{idowu2012fundamental} were known, but no product formula with polynomial exponents was available beyond the $n=1$ level (Remark~\ref{rem:catalan-beta}).

\paragraph{Headline products.}
Four evaluations illustrate the scope; in each product $k$ is the block index and $n$ is a fixed parameter (the generalised Eulerian level in~\eqref{E: intro beta}):
{\small
\begin{align}
\prod_{k=1}^{\infty}\left[\frac{(2k)^{2}(2k+3)}{(2k-1)(2k+2)^{2}}\right]^{k}
&=\frac{\pi}{2}, \label{E: intro wallis-type2}\\
\prod_{k=1}^{\infty}\left[\frac{(4k-2)^{2}(4k+1)(4k+3)}{(4k-3)(4k-1)(4k+2)^{2}}\right]^{k}
&=\sqrt{2}, \label{E: intro root2-type2}\\
\prod_{k=1}^{\infty}\left[\frac{(2k)^{2}(2k+1)^{2}(2k+4)^{6}(2k+7)}{(2k-1)(2k+2)^{6}(2k+5)^{2}(2k+6)^{2}}\right]^{k^{3}}
&=\frac{3645}{1048576}\,\pi^{6}, \label{E: intro pi6}\\
\prod_{k=0}^{\infty}\Big[\frac{(4k+3)^{T(n,0)}\cdots(4k+4n+3)^{T(n,n)}}{(4k+1)^{T(n,n)}\cdots(4k+4n+1)^{T(n,0)}}\Big]^{\binom{n+k}{n}}
&=e^{\beta'(-n)}\ \text{(odd $n$)}, \label{E: intro beta}
\end{align}
}
Equations~\eqref{E: intro wallis-type2} and~\eqref{E: intro root2-type2} are Type-2 products with exponent~$k$, proved in Section~\ref{sec:type2};~\eqref{E: intro pi6} has exponent $k^{3}$ and is assembled from a Type-4 ladder in Section~\ref{sec:poly-powers}. The Wallis--Eulerian product~\eqref{E: intro beta}, with weights $T(n,k)$ the coefficients of the generalised Eulerian polynomial $A_{n,4}(x)$~\cite{xiong2013general,oeisA225118,luschny2022eulerian}, is the main theorem, proved in Part~II for odd~$n$; for even~$n$ the raw product diverges. The chain~\eqref{E: intro wallis-type2} $\to$~\eqref{E: intro pi6} $\to$~\eqref{E: intro beta} is genuinely one construction applied with increasing ambition: the same power-sum mechanism underlies all three.

\paragraph{Positive vs.\ negative integers.}
At positive integers the beta derivative already admits a convergent product~\cite{weisstein2002beta}:
\[
e^{\beta'(n)}=\prod_{k=0}^{\infty}\frac{(4k+3)^{1/(4k+3)^{n}}}{(4k+1)^{1/(4k+1)^{n}}}.
\]
No analogous naive product exists at negative integers because the defining series diverges; the $T(n,k)$-weighted block form~\eqref{E: intro beta} fills that gap. Duke and Imamo\u{g}lu~\cite{article:duke2006} expressed the combination $\beta'(-n)+(\log 4)\beta(-n)$ as a finite sum of log-multiple-gamma terms at $1/4$ and~$3/4$; Part~II proves the complementary \emph{product side} and matches the two.

\paragraph{Roadmap.}
Part~I (Sections~\ref{sec:type1}--\ref{sec:poly-powers}) is self-contained and develops the product-design toolkit in order of increasing type: Type-1 products via the gamma function (Section~\ref{sec:type1}), Type-2 products via the Barnes $G$-function, including the exponential Catalan product $e^{2K/\pi}$ that first signalled the beta connection (Section~\ref{sec:type2}), the Type-$N$ theorem in standard and shifted form together with the analogue finder that lifts any Type-1 value to every type (Section~\ref{sec:products-typeN}), and the polynomial-exponent constructions $k^{2}$, $k^{3}$, $\pi^{M!}$ with an explicit design recipe (Section~\ref{sec:poly-powers}).

Part~II (Sections~\ref{sec:main-results}--\ref{sec:Ln}) proves the beta identity. Section~\ref{sec:main-results} states the headline results; Section~\ref{sec:proof-map} records the coefficient target and five-step spine; the proofs occupy Sections~\ref{sec:prelims}--\ref{sec:Ln} in dependency order. The product side is the shifted Type-$(N+1)$ construction of Section~\ref{sec:products-typeN} at quarter-integer slots with Eulerian weights; the new ingredient beyond Part~I is a closed-form evaluation of the functional-equation (FE) coefficients in the quarter-integer basis (Lemma~\ref{lem:A-closed}), which converts the template into Eulerian--binomial sums identified with Duke's polynomials~$P_{n+1,\ell}$ (Lemma~\ref{lem:T-binom-P}).

\part{Designing infinite products with polynomial exponents}\label{part:products}


\section{Type-1 products via the gamma function}\label{sec:type1}

Throughout Part I, $i\in\mathbb{N}_{+}$ is a fixed number of \emph{slots} and $(a_{\ell})_{\ell=1}^{i}$, $(b_{\ell})_{\ell=1}^{i}$ are finite slot sequences. A \emph{block} is the rational function
\begin{align}\label{E: p1 block}
\frac{(k+a_{1})\cdots(k+a_{i})}{(k+b_{1})\cdots(k+b_{i})}
\end{align}
of the product index~$k$, and we study infinite products of blocks raised to exponents that grow polynomially in~$k$. The theme of Part I is that \emph{moment constraints on the slots buy polynomial exponents}: each additional matching power sum $\sum_{\ell}a_{\ell}^{j}=\sum_{\ell}b_{\ell}^{j}$ raises by one the degree of exponent for which the product converges, and the limit is always an explicit ratio of multiple gamma values. We begin at degree zero.

\begin{definition}[Gamma function]\label{def:gamma}
The gamma function is the unique function with $\Gamma(z+1)=z\Gamma(z)$, $\Gamma(1)=1$, and $\log\Gamma$ log-convex on $(0,\infty)$, that is, $\log\Gamma$ is convex on $(0,\infty)$ (Bohr--Mollerup~\cite{article:9}). Since $1/\Gamma$ is entire, the Weierstrass factorisation
\begin{align}\label{E: p1 gamma weier}
\Gamma(z)=\frac{e^{-\gamma z}}{z}\prod_{k=1}^{\infty}\Big(1+\frac{z}{k}\Big)^{-1}e^{z/k}
\end{align}
holds for all $z\in\mathbb{C}\setminus\mathbb{Z}_{\leq 0}$~\cite{article:8}. We use freely the functional equation and Euler's reflection formula,
\begin{align}\label{E: p1 gamma fe}
\Gamma(z+1)=z\Gamma(z),\qquad
\Gamma(z)\Gamma(1-z)=\frac{\pi}{\sin(\pi z)}\quad(z\notin\mathbb{Z}).
\end{align}
\end{definition}

\begin{theorem}[Type-1 product]\label{thm:type1}
If $a_{\ell},b_{\ell}\in\mathbb{C}\setminus\mathbb{Z}_{\leq 0}$ and $\sum_{\ell=1}^{i}a_{\ell}=\sum_{\ell=1}^{i}b_{\ell}$, then
\begin{align}\label{E: p1 type1}
\prod_{k=0}^{\infty}\frac{(k+a_{1})\cdots(k+a_{i})}{(k+b_{1})\cdots(k+b_{i})}
=\frac{\Gamma(b_{1})\cdots\Gamma(b_{i})}{\Gamma(a_{1})\cdots\Gamma(a_{i})}.
\end{align}
Equivalently, shifting every slot by one: if $a_{\ell},b_{\ell}\notin\mathbb{Z}_{<0}$,
\begin{align}\label{E: p1 type1 alt}
\prod_{k=1}^{\infty}\frac{(k+a_{1})\cdots(k+a_{i})}{(k+b_{1})\cdots(k+b_{i})}
=\frac{\Gamma(b_{1}+1)\cdots\Gamma(b_{i}+1)}{\Gamma(a_{1}+1)\cdots\Gamma(a_{i}+1)}.
\end{align}
Products of this form are called \emph{Type-1 products}.
\end{theorem}

\begin{proof}
From~\eqref{E: p1 gamma weier}, for $z\notin\mathbb{Z}_{<0}$,
\[
\Gamma(z+1)=z\Gamma(z)=e^{-\gamma z}\prod_{k=1}^{\infty}\Big(1+\frac{z}{k}\Big)^{-1}e^{z/k}.
\]
Writing $T_{1}:=\sum_{\ell}a_{\ell}-\sum_{\ell}b_{\ell}$ and multiplying the factorisations,
\[
\frac{\Gamma(b_{1}+1)\cdots\Gamma(b_{i}+1)}{\Gamma(a_{1}+1)\cdots\Gamma(a_{i}+1)}
=e^{\gamma T_{1}}\prod_{k=1}^{\infty}\left\{\frac{(k+a_{1})\cdots(k+a_{i})}{(k+b_{1})\cdots(k+b_{i})}\,e^{-T_{1}/k}\right\},
\]
where each factor was grouped by the index~$k$ (legitimate since each individual product converges). If $T_{1}=0$, both exponential corrections disappear, giving~\eqref{E: p1 type1 alt}. Replacing every slot $a_{\ell}\mapsto a_{\ell}-1$, $b_{\ell}\mapsto b_{\ell}-1$ and reindexing $k\mapsto k+1$ gives~\eqref{E: p1 type1}.
\end{proof}

The manipulation of gamma values needed to evaluate the right side of~\eqref{E: p1 type1} is algorithmic: apply the functional equation until every argument lies in $(0,1]$, then substitute known values, most often $\Gamma(1/2)=\sqrt{\pi}$ and the reflection formula in~\eqref{E: p1 gamma fe}. The first example is the canonical one.

\begin{example}[Wallis product~\cite{ARTICLE:4,ARTICLE:5}]\label{ex:wallis}
\begin{align}\label{E: p1 wallis}
\prod_{k=1}^{\infty}\frac{(2k)^{2}}{(2k-1)(2k+1)}=\frac{\pi}{2}.
\end{align}
\end{example}

\begin{proof}
Theorem~\ref{thm:type1}, form~\eqref{E: p1 type1 alt}, with $i=2$ and $a_{1}=a_{2}=0$, $b_{1}=-\tfrac12$, $b_{2}=\tfrac12$ gives
\[
\prod_{k=1}^{\infty}\frac{k^{2}}{(k-\tfrac12)(k+\tfrac12)}=\frac{\Gamma(1/2)\,\Gamma(3/2)}{\Gamma(1)^{2}}
=\sqrt{\pi}\cdot\frac{\sqrt{\pi}}{2}=\frac{\pi}{2}. \qedhere
\]
\end{proof}

\begin{example}[Golden ratio]\label{ex:golden}
\begin{align}\label{E: p1 golden}
\prod_{k=0}^{\infty}\frac{(30k+9)(30k+21)}{(30k+5)(30k+25)}=\frac{1+\sqrt{5}}{2}=\varphi.
\end{align}
\end{example}

\begin{proof}
Theorem~\ref{thm:type1} with $a_{1}=\tfrac{3}{10}$, $a_{2}=\tfrac{7}{10}$, $b_{1}=\tfrac16$, $b_{2}=\tfrac56$ gives $\Gamma(1/6)\Gamma(5/6)/(\Gamma(3/10)\Gamma(7/10))$. By reflection this equals $\sin(3\pi/10)/\sin(\pi/6)=2\cos(\pi/5)=\varphi$.
\end{proof}

\begin{proposition}[Nested radicals]\label{prop:nested-radical}
For every $m\in\mathbb{N}_{+}$,
\begin{align}\label{E: p1 radical}
\prod_{k=0}^{\infty}\frac{(2^{m+1}k+2)(2^{m+1}k+2^{m+1}-2)}{(2^{m+1}k+1)(2^{m+1}k+2^{m+1}-1)}
=\underbrace{\sqrt{2+\sqrt{2+\sqrt{2+\cdots}}}}_{m}.
\end{align}
\end{proposition}

\begin{proof}
Theorem~\ref{thm:type1} with $a_{1}=2^{-m}$, $a_{2}=1-2^{-m}$, $b_{1}=2^{-(m+1)}$, $b_{2}=1-2^{-(m+1)}$ and reflection give
\[
\frac{\Gamma(2^{-(m+1)})\,\Gamma(1-2^{-(m+1)})}{\Gamma(2^{-m})\,\Gamma(1-2^{-m})}
=\frac{\sin(2^{-m}\pi)}{\sin(2^{-(m+1)}\pi)}
=2\cos\Big(\frac{\pi}{2^{m+1}}\Big).
\]
Iterating the double-angle identity $2\cos\theta=\sqrt{2+2\cos 2\theta}$ unwinds $2\cos(\pi/2^{m+1})$ into the $m$-fold nested radical.
\end{proof}

At $m=1$, Proposition~\ref{prop:nested-radical} is Catalan's 1875 product for $\sqrt{2}$~\cite{article:7}; at $m=2$ it recovers a product of Sondow and Yi~\cite{ARTICLE:3}. Rational targets are also designable.

\begin{example}[Rational values]\label{ex:rational-pq}
For $p,q\in\mathbb{N}_{+}$,
\begin{align}\label{E: p1 rational}
\prod_{k=1}^{\infty}\frac{(k+1)(qk+p-q)}{k\,(qk+p)}=\frac{p}{q}.
\end{align}
\end{example}

\begin{proof}
Theorem~\ref{thm:type1}, form~\eqref{E: p1 type1 alt}, with $a_{1}=1$, $a_{2}=\tfrac{p-q}{q}$, $b_{1}=0$, $b_{2}=\tfrac{p}{q}$: the value is $\Gamma(1)\Gamma(\tfrac{p}{q}+1)/(\Gamma(2)\Gamma(\tfrac{p}{q}))=p/q$.
\end{proof}

Since convergent products may be multiplied, Example~\ref{ex:rational-pq} combined with Example~\ref{ex:wallis} yields a Type-1 product for every positive rational multiple of~$\pi$, and more generally for any rational multiple of any power of~$\pi$. The reflection formula likewise gives Type-1 products for $\sin(p\pi/q)$ with $a_{1}=p/q$, $a_{2}=1-p/q$, $b_{1}=b_{2}=1/2$. The slot constants are cheap to design because only the \emph{first} moment is constrained; the next section adds a second moment constraint and, in return, an exponent.

\section{Type-2 products via the Barnes $G$-function}\label{sec:type2}

\begin{definition}[Barnes $G$-function]\label{def:barnesG}
The Barnes $G$-function~\cite{barnes1899gfunction} extends the superfactorial via $G(z+1)=\Gamma(z)G(z)$, $G(1)=1$, together with a convexity normalisation analogous to Bohr--Mollerup. Its Weierstrass form,
\begin{align}\label{E: p1 barnes weier}
G(z+1)=(2\pi)^{z/2}\exp\Big(-\frac{z+z^{2}(1+\gamma)}{2}\Big)
\prod_{k=1}^{\infty}\Big\{\Big(1+\frac{z}{k}\Big)^{k}\exp\Big(\frac{z^{2}}{2k}-z\Big)\Big\},
\end{align}
is valid for all $z\in\mathbb{C}$~\cite{article:10}.
\end{definition}

\begin{definition}[Two constants]\label{def:constants}
Catalan's constant is $K:=\sum_{k=0}^{\infty}(-1)^{k}(2k+1)^{-2}\approx 0.9159655942$ (written $K$ rather than $\mathbf{G}$ to avoid a clash with the Barnes function; this $K$ is not the product block index~$k$), and the Glaisher--Kinkelin constant is $A:=\exp(\tfrac{1}{12}-\zeta'(-1))\approx 1.2824271291$~\cite{article:12}.
\end{definition}

\begin{theorem}[Type-2 product]\label{thm:type2}
If $a_{\ell},b_{\ell}\in\mathbb{C}\setminus\mathbb{Z}_{<0}$, $\sum_{\ell}a_{\ell}=\sum_{\ell}b_{\ell}$, and $\sum_{\ell}a_{\ell}^{2}=\sum_{\ell}b_{\ell}^{2}$, then
\begin{align}\label{E: p1 type2}
\prod_{k=1}^{\infty}\left[\frac{(k+a_{1})\cdots(k+a_{i})}{(k+b_{1})\cdots(k+b_{i})}\right]^{k}
=\frac{G(a_{1}+1)\cdots G(a_{i}+1)}{G(b_{1}+1)\cdots G(b_{i}+1)}.
\end{align}
Products of this form are called \emph{Type-2 products}.
\end{theorem}

\begin{proof}
Set $T_{1}:=\sum_{\ell}a_{\ell}-\sum_{\ell}b_{\ell}$ and $T_{2}:=\sum_{\ell}a_{\ell}^{2}-\sum_{\ell}b_{\ell}^{2}$. Multiplying the factorisations~\eqref{E: p1 barnes weier} and grouping by the index~$k$,
\[
\frac{G(a_{1}+1)\cdots G(a_{i}+1)}{G(b_{1}+1)\cdots G(b_{i}+1)}
=(2\pi)^{T_{1}/2}\,e^{S}\prod_{k=1}^{\infty}\left\{\left[\frac{(k+a_{1})\cdots(k+a_{i})}{(k+b_{1})\cdots(k+b_{i})}\right]^{k}
\exp\Big(\frac{T_{2}}{2k}-T_{1}\Big)\right\},
\]
where $S=-\tfrac12\big(T_{1}+T_{2}(1+\gamma)\big)$. Under the hypotheses $T_{1}=T_{2}=0$, hence $S=0$ and every correction factor is~$1$.
\end{proof}

The evaluation routine parallels the Type-1 case: reduce all $G$-arguments to $(0,1]$ with $G(z+1)=\Gamma(z)G(z)$, then finish with gamma manipulations. The first example matters most for this paper.

\begin{example}[Exponential Catalan]\label{ex:catalan}
\begin{align}\label{E: p1 catalan}
\prod_{k=1}^{\infty}\left[\frac{(4k-1)^{3}(4k+3)}{(4k-3)(4k+1)^{3}}\right]^{k}=e^{2K/\pi}.
\end{align}
\end{example}

\begin{proof}
Theorem~\ref{thm:type2} with $i=4$ and
\[
a_{1}=a_{2}=a_{3}=-\tfrac14,\quad a_{4}=\tfrac34,\qquad
b_{1}=-\tfrac34,\quad b_{2}=b_{3}=b_{4}=\tfrac14
\]
(both moment constraints are immediate) gives
\[
\prod_{k=1}^{\infty}\left[\frac{(k-\tfrac14)^{3}(k+\tfrac34)}{(k-\tfrac34)(k+\tfrac14)^{3}}\right]^{k}
=\frac{G(3/4)^{3}\,G(7/4)}{G(1/4)\,G(5/4)^{3}}
=\frac{\Gamma(3/4)}{\Gamma(1/4)^{3}}\left(\frac{G(3/4)}{G(1/4)}\right)^{4},
\]
the last step by the functional equation. Adamchik's closed forms~\cite{article:6}
\[
G(1/4)=A^{-9/8}\,\Gamma(1/4)^{-3/4}\,e^{3/32-K/4\pi},\qquad
G(3/4)=A^{-9/8}\,\Gamma(3/4)^{-1/4}\,e^{3/32+K/4\pi}
\]
give $G(3/4)/G(1/4)=\Gamma(1/4)^{3/4}\,\Gamma(3/4)^{-1/4}\,e^{K/2\pi}$, whence the ratio collapses to $e^{2K/\pi}$.
\end{proof}

\begin{remark}[Looking ahead]\label{rem:catalan-beta}
Since $\beta'(-1)=2K/\pi$~\cite{idowu2012fundamental}, Example~\ref{ex:catalan} evaluates $e^{\beta'(-1)}$. The product~\eqref{E: p1 catalan} first appeared in the author's preprint~\cite{farrell2019wallis}. Products for the equivalent constants $e^{4K/\pi}$ and $e^{2K/\pi\pm1/2}$, in renormalised form with trailing correction factors, were given earlier by Kachi and Tzermias~\cite{kachi2012products}, in the circle of ideas of Guillera and Sondow~\cite{guillera2008double}. Part~II proves that this is the $n=1$ instance of a family of block products for $e^{\beta'(-n)}$ at every level~$n$; this single Type-2 computation is what first suggested that family.
\end{remark}

\begin{example}[A rational Type-2 value]\label{ex:type2-rational}
\begin{align}\label{E: p1 type2 rational}
\prod_{k=1}^{\infty}\left[\frac{(3k+1)(3k+5)^{2}(3k+7)}{(3k+2)(3k+4)^{2}(3k+8)}\right]^{k}=\frac{4}{5}.
\end{align}
\end{example}

\begin{proof}
Theorem~\ref{thm:type2} with $a_{1}=\tfrac13$, $a_{2}=a_{3}=\tfrac53$, $a_{4}=\tfrac73$ and $b_{1}=\tfrac23$, $b_{2}=b_{3}=\tfrac43$, $b_{4}=\tfrac83$, then repeated use of both functional equations:
\[
\frac{G(4/3)G(8/3)^{2}G(10/3)}{G(5/3)G(7/3)^{2}G(11/3)}
=\frac{\Gamma(7/3)\,G(4/3)\,G(8/3)}{\Gamma(8/3)\,G(5/3)\,G(7/3)}
=\frac{\Gamma(7/3)\,\Gamma(5/3)}{\Gamma(8/3)\,\Gamma(4/3)}
=\frac{4}{5}. \qedhere
\]
\end{proof}

The recursion $G(z+1)=\Gamma(z)G(z)$ also lets a Type-2 product be unfolded into a double product of Type-1 blocks, which is often the fastest route to a closed form.

\begin{proposition}[Double product]\label{prop:double-product}
Under the hypotheses of Theorem~\ref{thm:type2},
\begin{align}\label{E: p1 double}
\prod_{k=1}^{\infty}\left[\frac{(k+a_{1})\cdots(k+a_{i})}{(k+b_{1})\cdots(k+b_{i})}\right]^{k}
=\prod_{r=0}^{\infty}\prod_{k=1}^{\infty}\frac{(k+r+a_{1})\cdots(k+r+a_{i})}{(k+r+b_{1})\cdots(k+r+b_{i})}.
\end{align}
\end{proposition}

\begin{proof}
The factor with index $k$ appears with exponent $1$ in each inner product with $r<k$, i.e.\ $k$ times in total; commuting the absolutely convergent double product gives~\eqref{E: p1 double}.
\end{proof}

\begin{example}[$\sqrt{2}$ as a Type-2 product]\label{ex:root2-type2}
\begin{align}\label{E: p1 root2}
\prod_{k=1}^{\infty}\left[\frac{(4k-2)^{2}(4k+1)(4k+3)}{(4k-3)(4k-1)(4k+2)^{2}}\right]^{k}=\sqrt{2}.
\end{align}
\end{example}

\begin{proof}
With $a_{1}=a_{2}=-\tfrac12$, $a_{3}=\tfrac14$, $a_{4}=\tfrac34$ and $b_{1}=-\tfrac34$, $b_{2}=-\tfrac14$, $b_{3}=b_{4}=\tfrac12$, Proposition~\ref{prop:double-product} and Theorem~\ref{thm:type1} evaluate the inner products:
\begin{align*}
\prod_{r=0}^{\infty}\prod_{k=1}^{\infty}\frac{(k+r-\tfrac12)^{2}(k+r+\tfrac14)(k+r+\tfrac34)}{(k+r-\tfrac34)(k+r-\tfrac14)(k+r+\tfrac12)^{2}}
&=\prod_{r=0}^{\infty}\frac{\Gamma(r+\tfrac14)\,\Gamma(r+\tfrac34)\,\Gamma(r+\tfrac32)^{2}}{\Gamma(r+\tfrac12)^{2}\,\Gamma(r+\tfrac54)\,\Gamma(r+\tfrac74)}\\
&=\prod_{r=0}^{\infty}\frac{(4r+2)^{2}}{(4r+1)(4r+3)},
\end{align*}
using the gamma functional equation to cancel; the last product is Catalan's product for $\sqrt{2}$ (Proposition~\ref{prop:nested-radical} at $m=1$).
\end{proof}

Because $\Gamma(z)=G(z+1)/G(z)$, every Type-1 evaluation can be re-expressed through the $G$-function, and in fact every Type-1 product has a Type-2 \emph{analogue}: a genuinely different product, with exponent~$k$, converging to the same constant.

\begin{proposition}[Type-1 to Type-2 analogue finder]\label{prop:gamma-barnes}
If $a_{\ell},b_{\ell}\in\mathbb{C}\setminus\mathbb{Z}_{<0}$ and $\sum_{\ell}a_{\ell}=\sum_{\ell}b_{\ell}$, then
\begin{align}\label{E: p1 analogue2}
\prod_{k=1}^{\infty}\frac{(k+a_{1})\cdots(k+a_{i})}{(k+b_{1})\cdots(k+b_{i})}
=\prod_{k=1}^{\infty}\left[\frac{(k+a_{1})\cdots(k+a_{i})\,(k+b_{1}+1)\cdots(k+b_{i}+1)}{(k+a_{1}+1)\cdots(k+a_{i}+1)\,(k+b_{1})\cdots(k+b_{i})}\right]^{k}.
\end{align}
\end{proposition}

\begin{proof}
The right side is a Type-2 product with slot families $A'=\{a_{\ell}\}\cup\{b_{\ell}+1\}$ and $B'=\{a_{\ell}+1\}\cup\{b_{\ell}\}$. First moments agree because $\sum a_{\ell}+\sum(b_{\ell}+1)=\sum b_{\ell}+\sum(a_{\ell}+1)$; second moments agree because
\[
\sum_{\ell}a_{\ell}^{2}+\sum_{\ell}(b_{\ell}+1)^{2}
-\sum_{\ell}b_{\ell}^{2}-\sum_{\ell}(a_{\ell}+1)^{2}
=2\Big(\sum_{\ell}b_{\ell}-\sum_{\ell}a_{\ell}\Big)=0.
\]
By Theorem~\ref{thm:type2} and $\Gamma(z)=G(z+1)/G(z)$, its value is
\[
\prod_{\ell=1}^{i}\frac{G(a_{\ell}+1)\,G(b_{\ell}+2)}{G(a_{\ell}+2)\,G(b_{\ell}+1)}
=\prod_{\ell=1}^{i}\frac{\Gamma(b_{\ell}+1)}{\Gamma(a_{\ell}+1)},
\]
which is the Type-1 value on the left of~\eqref{E: p1 analogue2} by Theorem~\ref{thm:type1}.
\end{proof}

\begin{example}[Type-2 Wallis product]\label{ex:wallis-type2}
\begin{align}\label{E: p1 wallis2}
\prod_{k=1}^{\infty}\left[\frac{(2k)^{2}(2k+3)}{(2k-1)(2k+2)^{2}}\right]^{k}=\frac{\pi}{2}.
\end{align}
\end{example}

\begin{proof}
Proposition~\ref{prop:gamma-barnes} with the Wallis constants $a_{1}=a_{2}=0$, $b_{1}=-\tfrac12$, $b_{2}=\tfrac12$ of Example~\ref{ex:wallis}: after cancelling the common slot $\tfrac12$, the right side of~\eqref{E: p1 analogue2} is exactly the displayed block.
\end{proof}

Type-2 products already reach constants of every flavour of irrationality: Example~\ref{ex:type2-rational} is rational, Example~\ref{ex:root2-type2} algebraic irrational, Example~\ref{ex:wallis-type2} transcendental, and the arithmetic nature of $e^{2K/\pi}$ in Example~\ref{ex:catalan} is unknown. The reverse of Proposition~\ref{prop:gamma-barnes} fails in general: unfolding Example~\ref{ex:catalan} by Proposition~\ref{prop:double-product} leaves the residual product $\prod_{r\geq 0}\Gamma(r+\tfrac14)\Gamma(r+\tfrac54)^{3}/(\Gamma(r+\tfrac34)^{3}\Gamma(r+\tfrac74))$, whose partial quotients do not telescope to gamma values; some Type-2 values appear to be genuinely inaccessible at Type-1.

\section{Type-$N$ products and analogue lifting}\label{sec:products-typeN}

\begin{definition}[Multiple gamma function]\label{def:multiple-gamma}
Vign\'eras' multiple gamma functions $\Gamma_{n}$ are determined by the functional equation and normalisation
\begin{align}\label{E: p1 mgf fe}
\Gamma_{n+1}(z+1)=\frac{\Gamma_{n+1}(z)}{\Gamma_{n}(z)},\qquad
\Gamma_{n}(1)=1,\qquad \Gamma_{1}=\Gamma,
\end{align}
together with a log-convexity condition~\cite{ARTICLE:2}. Equivalently, $\log\Gamma_{n+1}(z)=\log\Gamma_{n+1}(z-1)-\log\Gamma_{n}(z-1)$ (the shift law used in Part~II). In particular $G(z)=\Gamma_{2}(z)^{-1}$. Some references use Barnes' generalised $G$-functions $G_{n}$ instead; the translation is $\Gamma_{n}(z)=G_{n}(z)^{(-1)^{n+1}}$ and nothing below depends on the choice.
\end{definition}

The pattern of Theorems~\ref{thm:type1} and~\ref{thm:type2} continues: matching power sums through order~$n$ buys the binomial exponent $\binom{n+k-2}{n-1}$, a polynomial in~$k$ of degree $n-1$, and the limit is a $\Gamma_{n}$-ratio. Part~(ii) below is the \emph{shifted} form, which starts the product at $k=0$ and trades one extra matching power sum for the next binomial exponent; it is the form used for the beta products of Part~II.

\begin{theorem}[Type-$N$ product]\label{thm:typeN}
Let $n\in\mathbb{N}$ and let $(a_{\ell})_{\ell=1}^{i}$, $(b_{\ell})_{\ell=1}^{i}$ be finite slot sequences.
\begin{enumerate}
\item[(i)] If $n\geq 1$, $a_{\ell},b_{\ell}\notin\mathbb{Z}_{<0}$, and $\sum_{\ell}a_{\ell}^{j}=\sum_{\ell}b_{\ell}^{j}$ for $j=1,\ldots,n$, then
\begin{align}\label{E: typeN product}
\prod_{k=1}^{\infty}\left[\frac{(k+a_{1})\cdots(k+a_{i})}{(k+b_{1})\cdots(k+b_{i})}\right]^{\binom{n+k-2}{n-1}}
=\frac{\Gamma_{n}(b_{1}+1)\cdots\Gamma_{n}(b_{i}+1)}{\Gamma_{n}(a_{1}+1)\cdots\Gamma_{n}(a_{i}+1)}.
\end{align}
\item[(ii)] If $n\geq 0$, $a_{\ell},b_{\ell}\notin\mathbb{Z}_{\leq 0}$, and $\sum_{\ell}a_{\ell}^{j}=\sum_{\ell}b_{\ell}^{j}$ for $j=1,\ldots,n+1$, then
\begin{align}\label{E: typeN shift}
\prod_{k=0}^{\infty}\left[\frac{(k+a_{1})\cdots(k+a_{i})}{(k+b_{1})\cdots(k+b_{i})}\right]^{\binom{n+k}{n}}
=\frac{\Gamma_{n+1}(b_{1})\cdots\Gamma_{n+1}(b_{i})}{\Gamma_{n+1}(a_{1})\cdots\Gamma_{n+1}(a_{i})}.
\end{align}
\end{enumerate}
Products of the form~\eqref{E: typeN product} are called \emph{Type-$N$ products} (with $N=n$).
\end{theorem}

\begin{proof}
(i) Choi's Weierstrass canonical form~\cite{ARTICLE:2} reads, for $z\notin\mathbb{Z}_{<0}$,
\begin{align}\label{E: p1 choi weier}
\Gamma_{n}(z+1)=e^{Q_{n}(z)}\prod_{k=1}^{\infty}\left\{\Big(1+\frac{z}{k}\Big)^{-\binom{n+k-2}{n-1}}
\exp\left[\binom{n+k-2}{n-1}\sum_{j=1}^{n}\frac{(-1)^{j-1}}{j}\frac{z^{j}}{k^{j}}\right]\right\},
\end{align}
where $Q_{n}$ is a polynomial of degree at most~$n$ with $Q_{n}(0)=0$ (it is built from the Bernoulli-type polynomials $p_{m}(z)=\frac{1}{m+1}\sum_{t=1}^{m+1}\binom{m+1}{t}B_{m+1-t}z^{t}$ with $m\leq n-1$, each of which has no constant term; see~\cite{ARTICLE:2} for the explicit expression). Multiplying the factorisations and grouping factors of equal index~$k$,
\begin{multline*}
\frac{\Gamma_{n}(b_{1}+1)\cdots\Gamma_{n}(b_{i}+1)}{\Gamma_{n}(a_{1}+1)\cdots\Gamma_{n}(a_{i}+1)}
=e^{S}\prod_{k=1}^{\infty}\Biggl\{\left[\frac{(k+a_{1})\cdots(k+a_{i})}{(k+b_{1})\cdots(k+b_{i})}\right]^{\binom{n+k-2}{n-1}}\\
\times\exp\Biggl[\binom{n+k-2}{n-1}\sum_{j=1}^{n}T_{j,k}\Biggr]\Biggr\},
\end{multline*}
with
\begin{align*}
T_{j,k}&=\frac{(-1)^{j-1}}{j\,k^{j}}\Big(\sum_{\ell}b_{\ell}^{j}-\sum_{\ell}a_{\ell}^{j}\Big),
\\
S&=\sum_{\ell}Q_{n}(b_{\ell})-\sum_{\ell}Q_{n}(a_{\ell}).
\end{align*}
Under the hypotheses every $T_{j,k}$ vanishes; and since $Q_{n}$ is a polynomial of degree $\leq n$ with no constant term, $S$ is a linear combination of the differences $\sum b_{\ell}^{j}-\sum a_{\ell}^{j}$ for $1\leq j\leq n$, hence $S=0$. All exponential corrections disappear, giving~\eqref{E: typeN product}.

(ii) Reindex $k\mapsto k+1$ in~\eqref{E: typeN shift}: the left side becomes
\[
\prod_{k=1}^{\infty}\left[\frac{(k+(a_{1}-1))\cdots(k+(a_{i}-1))}{(k+(b_{1}-1))\cdots(k+(b_{i}-1))}\right]^{\binom{(n+1)+k-2}{(n+1)-1}},
\]
a Type-$(n+1)$ product with slots $a_{\ell}-1$, $b_{\ell}-1\notin\mathbb{Z}_{<0}$. By the binomial theorem, $\sum_{\ell}(a_{\ell}-1)^{j}=\sum_{\ell}(b_{\ell}-1)^{j}$ for $j=1,\ldots,n+1$ follows from the hypotheses. Part~(i) at level $n+1$ evaluates the product as $\prod_{\ell}\Gamma_{n+1}(b_{\ell}-1+1)/\Gamma_{n+1}(a_{\ell}-1+1)$, which is~\eqref{E: typeN shift}. At $n=0$, part~(ii) is Theorem~\ref{thm:type1}, form~\eqref{E: p1 type1}.
\end{proof}

\begin{remark}[Signed slot families]\label{rem:signed}
We often apply Theorem~\ref{thm:typeN} to slot families carrying \emph{signed multiplicities}: a slot with multiplicity $-1$ on the $a$-side is simply a slot with multiplicity $+1$ on the $b$-side, so nothing new is being asserted; only the bookkeeping changes. Power-sum constraints are then read with signs.
\end{remark}

Finding slot constants that match power sums through order~$n$ is exactly the Prouhet--Tarry--Escott problem~\cite{article:16}, so ideal solutions become scarce as $n$ grows, and closed forms for $\Gamma_{n}$ with $n\geq 3$ are scarcer still (already the duplication formula for $G$~\cite{ARTICLE:1} has no tidy analogue one level up). The way around both problems is systematic: \emph{lift} a known low-type evaluation to any higher type. The next lemma provides the required power sums; it is a standard fact in the Prouhet--Tarry--Escott circle of ideas, proved here for completeness.

\begin{lemma}[Binomial power-sum criteria]\label{lem:pte}
Let $n\in\mathbb{N}_{+}$ and let $A=\{a_{1},\ldots,a_{i}\}$, $B=\{b_{1},\ldots,b_{i}\}$ be multisets with $\sum_{a\in A}a=\sum_{b\in B}b$. Then for every $m\in\{1,\ldots,n+1\}$,
\begin{align}\label{E: p1 pte}
\sum_{r=0}^{n}\binom{n}{r}(-1)^{r}\sum_{a\in A}(a+r)^{m}
=\sum_{r=0}^{n}\binom{n}{r}(-1)^{r}\sum_{b\in B}(b+r)^{m}.
\end{align}
\end{lemma}

\begin{proof}
Write $S_{r}:=\sum_{t=0}^{n}\binom{n}{t}(-1)^{t}t^{r}$. Euler showed $S_{r}=0$ for $0\leq r<n$ and $S_{n}=(-1)^{n}n!$~\cite{euler1755institutiones,gould1978euler}. Expanding by the binomial theorem and exchanging sums,
\[
\sum_{r=0}^{n}\binom{n}{r}(-1)^{r}\sum_{a\in A}(a+r)^{m}
=\sum_{a\in A}\sum_{t=0}^{m}\binom{m}{t}a^{m-t}\,S_{t}.
\]
For $m<n$ every $S_{t}$ with $t\leq m$ vanishes and both sides of~\eqref{E: p1 pte} are zero. For $m=n$ only $t=n$ survives, leaving $i\,S_{n}$, the same for $A$ and $B$. For $m=n+1$ the terms $t=n$ and $t=n+1$ survive, leaving
\[
(n+1)\,S_{n}\sum_{a\in A}a\;+\;i\,S_{n+1},
\]
which agrees for $A$ and $B$ precisely because $\sum_{a\in A}a=\sum_{b\in B}b$.
\end{proof}

\begin{theorem}[Type-$N$ analogue finder]\label{thm:analogue-finder}
Let $n\in\mathbb{N}_{+}$ and let $A=\{a_{1},\ldots,a_{i}\}$, $B=\{b_{1},\ldots,b_{i}\}\subset\mathbb{C}\setminus\mathbb{Z}_{<0}$ satisfy $\sum_{a\in A}a=\sum_{b\in B}b$. Then
\begin{align}\label{E: p1 analogueN}
&\prod_{k=1}^{\infty}\prod_{r=0}^{n-1}\prod_{\ell=1}^{i}
\left(\frac{k+a_{\ell}+r}{k+b_{\ell}+r}\right)^{(-1)^{r}\binom{n-1}{r}\binom{n+k-2}{n-1}}
\nonumber\\
&\qquad=\frac{\Gamma(b_{1}+1)\cdots\Gamma(b_{i}+1)}{\Gamma(a_{1}+1)\cdots\Gamma(a_{i}+1)}
=\prod_{k=1}^{\infty}\frac{(k+a_{1})\cdots(k+a_{i})}{(k+b_{1})\cdots(k+b_{i})}.
\end{align}
Thus every Type-1 value acquires a Type-$n$ product for every $n$.
\end{theorem}

\begin{proof}
Iterating $\Gamma_{m}(z)=\Gamma_{m+1}(z)/\Gamma_{m+1}(z+1)$ (which is~\eqref{E: p1 mgf fe} rearranged) gives, by induction on $n\geq 1$,
\begin{align}\label{E: p1 telescope}
\Gamma(z)=\prod_{r=0}^{n-1}\Gamma_{n}(z+r)^{(-1)^{r}\binom{n-1}{r}},
\end{align}
the induction step being Pascal's rule on the exponents after substituting $\Gamma_{n}(z+r)=\Gamma_{n+1}(z+r)/\Gamma_{n+1}(z+r+1)$.

Now consider the signed slot family $\{a_{\ell}+r\}$ versus $\{b_{\ell}+r\}$, each slot weighted by $(-1)^{r}\binom{n-1}{r}$ (Remark~\ref{rem:signed}). Lemma~\ref{lem:pte}, applied with $n-1$ in place of~$n$, shows that the signed power sums of the two families agree for $j=1,\ldots,n$, which is the hypothesis of Theorem~\ref{thm:typeN}(i) at level~$n$. Hence the left side of~\eqref{E: p1 analogueN} equals
\[
\prod_{r=0}^{n-1}\prod_{\ell=1}^{i}
\left(\frac{\Gamma_{n}(b_{\ell}+r+1)}{\Gamma_{n}(a_{\ell}+r+1)}\right)^{(-1)^{r}\binom{n-1}{r}}
=\prod_{\ell=1}^{i}\frac{\Gamma(b_{\ell}+1)}{\Gamma(a_{\ell}+1)},
\]
by~\eqref{E: p1 telescope} at $z=b_{\ell}+1$ and $z=a_{\ell}+1$. The second equality in~\eqref{E: p1 analogueN} is Theorem~\ref{thm:type1}.
\end{proof}

\begin{example}[Wallis at every level]\label{ex:typeN-wallis}
Take $A=\{0,0\}$, $B=\{-\tfrac12,\tfrac12\}$ in Theorem~\ref{thm:analogue-finder}:
\begin{align}\label{E: p1 wallis every level}
\prod_{k=1}^{\infty}\prod_{r=0}^{n-1}\left(\frac{(2k+2r)^{2}}{(2k+2r-1)(2k+2r+1)}\right)^{(-1)^{r}\binom{n-1}{r}\binom{n+k-2}{n-1}}=\frac{\pi}{2}
\qquad\text{for every }n\geq 1.
\end{align}
Collecting the inner factors into a single block, the first three levels read
\begin{align*}
n=1&:\quad \prod_{k=1}^{\infty}\frac{(2k)^{2}}{(2k-1)(2k+1)}=\frac{\pi}{2},\\
n=2&:\quad \prod_{k=1}^{\infty}\left[\frac{(2k)^{2}(2k+3)}{(2k-1)(2k+2)^{2}}\right]^{k}=\frac{\pi}{2},\\
n=3&:\quad \prod_{k=1}^{\infty}\left[\frac{(2k)^{2}(2k+1)(2k+3)(2k+4)^{2}}{(2k-1)(2k+2)^{4}(2k+5)}\right]^{\frac{1}{2}k(k+1)}=\frac{\pi}{2}.
\end{align*}
The levels $n=1,2$ recover Examples~\ref{ex:wallis} and~\ref{ex:wallis-type2}. Numerically the speed of convergence degrades as $n$ grows (the partial products converge sublinearly), the price paid for the higher-degree exponent.
\end{example}

\section{Polynomial exponents by design}\label{sec:poly-powers}

The binomial exponents $\binom{n+k-2}{n-1}$ ($n=1,2,3,\ldots$) are integer-valued polynomials in~$k$ of every degree, and they form a $\mathbb{Z}$-basis of the space of integer-valued polynomials: the transition matrix to Newton's forward-difference basis $\binom{k-1}{d}$ is unitriangular with integer entries (Vandermonde). So if a \emph{single} block admits convergent products at several consecutive types, integer combinations of the exponents (multiplying and dividing the evaluated products) realise \emph{any} integer-valued polynomial exponent whose degree the type ladder covers. This section carries the programme out for the Wallis block ladder of Example~\ref{ex:typeN-wallis}, producing exponents $k^{2}$, $k^{3}$, and in general $k^{M}$.

One block, several types, first.

\begin{example}[Same block, different exponents]\label{ex:same-block}
Let
\[
W(k):=\frac{(2k)^{2}(2k+1)(2k+3)(2k+4)^{2}}{(2k-1)(2k+2)^{4}(2k+5)},
\]
the level-$3$ block of Example~\ref{ex:typeN-wallis}, with slot families $a=\{0,0,\tfrac12,\tfrac32,2,2\}$ and $b=\{-\tfrac12,1,1,1,1,\tfrac52\}$ (power sums match through order~$3$). Then
\begin{align}\label{E: p1 same block}
\prod_{k=1}^{\infty}W(k)=\frac{5}{4},\qquad
\prod_{k=1}^{\infty}W(k)^{k}=\frac{4}{3},\qquad
\prod_{k=1}^{\infty}W(k)^{\frac{1}{2}k(k+1)}=\frac{\pi}{2}.
\end{align}
\end{example}

\begin{proof}
The third product is Example~\ref{ex:typeN-wallis} at $n=3$. The first is Type-1 (Theorem~\ref{thm:type1}):
\[
\frac{\Gamma(\tfrac12)\Gamma(2)^{4}\Gamma(\tfrac72)}{\Gamma(1)^{2}\Gamma(\tfrac32)\Gamma(\tfrac52)\Gamma(3)^{2}}
=\frac{\sqrt{\pi}\cdot\tfrac{15\sqrt{\pi}}{8}}{\tfrac{\sqrt{\pi}}{2}\cdot\tfrac{3\sqrt{\pi}}{4}\cdot 4}
=\frac{5}{4}.
\]
The second is Type-2 (Theorem~\ref{thm:type2}); with $G(1)=G(2)=G(3)=1$,
\[
\frac{G(\tfrac32)G(\tfrac52)}{G(\tfrac12)G(\tfrac72)}
=\frac{G(\tfrac32)\,\Gamma(\tfrac32)G(\tfrac32)}{G(\tfrac12)\,\Gamma(\tfrac52)\Gamma(\tfrac32)G(\tfrac32)}
=\frac{\Gamma(\tfrac12)G(\tfrac12)}{G(\tfrac12)\,\Gamma(\tfrac52)}
=\frac{\Gamma(\tfrac12)}{\Gamma(\tfrac52)}=\frac{4}{3}. \qedhere
\]
\end{proof}

The three exponents $1$, $k$, $\tfrac12 k(k+1)$ in~\eqref{E: p1 same block} span the integer-valued polynomials of degree $\leq 2$, so any such exponent is now available on the block~$W$: raise the three evaluations to the corresponding integer powers and multiply. To push to degree three we need the same ladder one level up, and for that the following shift lemma converts a type ladder on one block into a type ladder on its \emph{difference block}.

\begin{lemma}[Shift lemma]\label{lem:shift-type}
Let $U(k)$ be a block~\eqref{E: p1 block} whose slot family matches power sums through order $m-1$ for some $m\geq 2$, and set $V(k):=U(k)/U(k+1)$. Then the slot family of~$V$ matches power sums through order~$m$, and
\begin{align}\label{E: p1 shift}
\prod_{k=1}^{\infty}V(k)^{\binom{m+k-2}{m-1}}
=\prod_{k=1}^{\infty}U(k)^{\binom{m+k-3}{m-2}},
\end{align}
i.e.\ the Type-$m$ product of $V$ equals the Type-$(m-1)$ product of $U$, both absolutely convergent.
\end{lemma}

\begin{proof}
The slot family of $V$ is $(a\cup(b+1)\,;\,b\cup(a+1))$ where $(a;b)$ is the family of~$U$. For $1\leq j\leq m$, the $j$-th power-sum defect of the $V$-family is
\[
\Big[\sum_{\ell}a_{\ell}^{j}-\sum_{\ell}b_{\ell}^{j}\Big]
+\Big[\sum_{\ell}(b_{\ell}+1)^{j}-\sum_{\ell}(a_{\ell}+1)^{j}\Big]
=-\sum_{t=1}^{j-1}\binom{j}{t}\Big[\sum_{\ell}a_{\ell}^{t}-\sum_{\ell}b_{\ell}^{t}\Big]=0,
\]
because the top ($t=j$) binomial terms cancel the unshifted defect, the $t=0$ terms cancel by equal slot counts, and the remaining defects have order $t\leq j-1\leq m-1$. So the $V$-family matches through order~$m$; the left side of~\eqref{E: p1 shift} converges absolutely by Theorem~\ref{thm:typeN}(i), and the right side does likewise at level $m-1$.

Write $e_{m}(k):=\binom{m+k-2}{m-1}$, so $e_{m}(0)=0$ and, by Pascal's rule, $e_{m}(k)-e_{m}(k-1)=e_{m-1}(k)$. For finite~$K$,
\begin{align*}
\prod_{k=1}^{K}V(k)^{e_{m}(k)}
&=\prod_{k=1}^{K}U(k)^{e_{m}(k)}\prod_{k=2}^{K+1}U(k)^{-e_{m}(k-1)}\\
&=\left[\prod_{k=1}^{K}U(k)^{e_{m-1}(k)}\right]U(K+1)^{-e_{m}(K)}.
\end{align*}
Since the $U$-family matches power sums through $m-1$, $\log U(K+1)=O(K^{-m})$, while $e_{m}(K)=O(K^{m-1})$; therefore $U(K+1)^{-e_{m}(K)}\to 1$ and letting $K\to\infty$ gives~\eqref{E: p1 shift}.
\end{proof}

\begin{example}[Exponent $k^{3}$ and $\pi^{6}$]\label{ex:pi6}
\begin{align}\label{E: p1 pi6}
\prod_{k=1}^{\infty}\left[\frac{(2k)^{2}(2k+1)^{2}(2k+4)^{6}(2k+7)}{(2k-1)(2k+2)^{6}(2k+5)^{2}(2k+6)^{2}}\right]^{k^{3}}
=\frac{3645}{1048576}\,\pi^{6}.
\end{align}
\end{example}

\begin{proof}
Let $W$ be the block of Example~\ref{ex:same-block} and set $V(k):=W(k)/W(k+1)$; cancelling common factors gives exactly the displayed block (it is also the level-$4$ block of Example~\ref{ex:typeN-wallis}). By Lemma~\ref{lem:shift-type} and~\eqref{E: p1 same block},
\begin{align*}
\prod_{k=1}^{\infty}V(k)^{k}
&=\prod_{k=1}^{\infty}W(k)=\frac{5}{4},
\\
\prod_{k=1}^{\infty}V(k)^{\frac12 k(k+1)}
&=\prod_{k=1}^{\infty}W(k)^{k}=\frac{4}{3},
\\
\prod_{k=1}^{\infty}V(k)^{\frac16 k(k+1)(k+2)}
&=\frac{\pi}{2}.
\end{align*}
Now express the target exponent in the binomial basis:
\[
k^{3}=6\cdot\tfrac{1}{6}k(k+1)(k+2)-6\cdot\tfrac{1}{2}k(k+1)+k .
\]
Multiplying and dividing the three absolutely convergent products accordingly,
\[
\prod_{k=1}^{\infty}V(k)^{k^{3}}
=\Big(\frac{\pi}{2}\Big)^{6}\Big(\frac{4}{3}\Big)^{-6}\Big(\frac{5}{4}\Big)
=\frac{3645}{1048576}\,\pi^{6}. \qedhere
\]
\end{proof}

\begin{example}[Exponent $k^{2}$]\label{ex:k2}
With $W$ as in Example~\ref{ex:same-block},
\begin{align}\label{E: p1 k2}
\prod_{k=1}^{\infty}\left[\frac{(2k)^{2}(2k+1)(2k+3)(2k+4)^{2}}{(2k-1)(2k+2)^{4}(2k+5)}\right]^{k^{2}}
=\frac{3\pi^{2}}{16}.
\end{align}
\end{example}

\begin{proof}
The same multiply-and-divide method as in Example~\ref{ex:pi6}, one degree down: $k^{2}=2\cdot\tfrac12 k(k+1)-k$, so by~\eqref{E: p1 same block},
\[
\prod_{k=1}^{\infty}W(k)^{k^{2}}=\Big(\frac{\pi}{2}\Big)^{2}\Big(\frac{4}{3}\Big)^{-1}=\frac{3\pi^{2}}{16}. \qedhere
\]
\end{proof}

The same mechanism works at every degree, with the power of~$\pi$ determined by the leading binomial coefficient.

\begin{theorem}[$\pi^{M!}$ products]\label{thm:pi-powers}
For $n\geq 1$ let
\[
W_{n}(k):=\prod_{r=0}^{n-1}\left(\frac{(2k+2r)^{2}}{(2k+2r-1)(2k+2r+1)}\right)^{(-1)^{r}\binom{n-1}{r}}
\]
be the level-$n$ Wallis block of Example~\ref{ex:typeN-wallis}. For every $M\in\mathbb{N}_{+}$ there exist $a,b\in\mathbb{N}_{+}$ such that
\begin{align}\label{E: p1 pi powers}
\prod_{k=1}^{\infty}W_{M+1}(k)^{k^{M}}=\frac{a\,\pi^{M!}}{b}.
\end{align}
\end{theorem}

\begin{proof}
Pascal's rule on the weights $(-1)^{r}\binom{n-1}{r}$ gives the recursion $W_{n+1}(k)=W_{n}(k)/W_{n}(k+1)$, and the slot family of $W_{n}$ matches power sums through order~$n$ (Lemma~\ref{lem:pte} with $n-1$ in place of~$n$, as in Theorem~\ref{thm:analogue-finder}). Iterating Lemma~\ref{lem:shift-type},
\begin{align}\label{E: p1 ladder}
\prod_{k=1}^{\infty}W_{M+1}(k)^{\binom{j+k-2}{j-1}}
=\prod_{k=1}^{\infty}W_{M+2-j}(k),
\qquad 1\leq j\leq M+1 .
\end{align}
At $j=M+1$ the right side is the Wallis product $\pi/2$ (Example~\ref{ex:wallis}). For $j\leq M$, i.e.\ $m:=M+2-j\geq 2$, the right side is the Type-1 product of~$W_{m}$, which by Theorem~\ref{thm:type1} equals
\[
\prod_{r=0}^{m-1}\left[\frac{\Gamma(r+\tfrac12)\,\Gamma(r+\tfrac32)}{\Gamma(r+1)^{2}}\right]^{(-1)^{r}\binom{m-1}{r}},
\]
and each bracket is $(r+\tfrac12)\Gamma(r+\tfrac12)^{2}/(r!)^{2}$, a positive \emph{rational} multiple of~$\pi$. The total exponent of $\pi$ is $\sum_{r=0}^{m-1}(-1)^{r}\binom{m-1}{r}=0$ for $m\geq 2$, so the value is a positive rational.

Finally, since the exponents $\binom{j+k-2}{j-1}$ ($1\leq j\leq M+1$) form a $\mathbb{Z}$-basis of the integer-valued polynomials of degree~$\leq M$, there are integers $c_{1},\ldots,c_{M}$ with
\[
k^{M}=M!\binom{M+k-1}{M}+\sum_{j=1}^{M}c_{j}\binom{j+k-2}{j-1}.
\]
Raising the evaluations~\eqref{E: p1 ladder} to these integer powers and multiplying, the rational contributions combine into $a/b$ and the leading factor contributes $(\pi/2)^{M!}$, giving~\eqref{E: p1 pi powers}.
\end{proof}

Examples~\ref{ex:k2} and~\ref{ex:pi6} realised the cases $M=2,3$ of this mechanism on slightly tidier blocks; the powers $\pi^{2}$ and $\pi^{6}$ there are the $r=M!$ of Theorem~\ref{thm:pi-powers}.

\subsection{A design recipe}\label{sec:recipe}

The constructions above assemble into a method for designing an infinite product with a prescribed polynomial exponent and an explicitly evaluable limit:

\begin{enumerate}
\item \textbf{Choose slot constants.} Pick slot families whose power sums match through the order required; equivalently, a solution of the Prouhet--Tarry--Escott problem~\cite{article:16}. For a prescribed \emph{target constant}, start instead from a Type-1 or Type-2 product for it (Sections~\ref{sec:type1}--\ref{sec:type2}), where only one or two moments constrain the choice.
\item \textbf{Lift to the required type.} Apply the analogue finder (Theorem~\ref{thm:analogue-finder}) or, one level at a time, the difference-block construction $U(k)\mapsto U(k)/U(k+1)$ of Lemma~\ref{lem:shift-type}. This yields one block admitting convergent products at all types up to the target degree, with known values at every rung (Example~\ref{ex:same-block}).
\item \textbf{Combine exponents.} Express the desired integer-valued polynomial exponent in the binomial basis $\binom{j+k-2}{j-1}$ with integer coefficients, and multiply and divide the rung evaluations accordingly (Examples~\ref{ex:pi6} and~\ref{ex:k2}).
\item \textbf{Evaluate.} Reduce the resulting $\Gamma_{n}$-ratios with the functional equation~\eqref{E: p1 mgf fe}, reflection and duplication formulas~\cite{ARTICLE:1}, and the known closed forms for $G$ at rational arguments~\cite{article:6,article:14}.
\end{enumerate}

Any integer-valued polynomial in~$k$ is reachable as an exponent this way. Part~II applies exactly this design at quarter-integer slots: the block products for the Dirichlet beta function use the shifted form of Theorem~\ref{thm:typeN}(ii) with slots $(4k+3)/4$ and $(4(n-k)+1)/4$, weighted by generalised Eulerian numbers whose recurrence supplies the required power-sum matching.


\part{The Dirichlet beta function at negative integers}\label{part:beta}

\noindent The main theorem is the design of Part~I carried out at quarter-integer slots. The block at index~$k$ uses numerator slots $(4k+4j+3)/4$ and denominator slots $(4k+4(n-j)+1)/4$, each with multiplicity $T(n,j)$; the generalised Eulerian recurrence for the weights $T(n,j)$ supplies the power-sum matching through order~$n$ (and through $n+1$ when $n$ is odd) that Theorem~\ref{thm:typeN}(ii) requires, so for odd~$n$ the infinite product collapses to a finite $\Gamma_{n+1}$-ratio exactly as in Section~\ref{sec:products-typeN}. What is genuinely new is the identification of that finite gamma template with $\beta'(-n)+(\log 4)\beta(-n)$. Section~\ref{sec:proof-map} states the coefficient target and the five-step spine before the toolkit lemmas.

\section{Statement of main results}\label{sec:main-results}

\smallskip\noindent\textit{This section states the headline results. Every object used below is defined here; proofs are deferred to the sections indicated. The remainder of the supplement then follows proof order.}

\smallskip\noindent\textbf{Eulerian weights.} For $n\ge 0$ and $0\le k\le n$, let $T(n,k)$ be the coefficient of $x^{n-k}$ in the generalised Eulerian polynomial $A_{n,4}(x)$ (OEIS A225118 lists these rows in descending order)~\cite{xiong2013general,oeisA225118}. Equivalently $T(n,k)$ is the coefficient of $x^{k}$ in $x^{n}A_{n,4}(1/x)$, with $T(0,0)=1$, $T(n,k)=0$ for $k<0$ or $k>n$, and recurrence
\begin{align}\label{E: T recur}
T(n+1,k) \;=\; \big(4(n-k)+5\big)\,T(n,k-1) \;+\; \big(4k+3\big)\,T(n,k),
\qquad T(n,-1)=0.
\end{align}

\smallskip\noindent\textbf{Duke polynomials.} For $k,\ell\ge 1$, Duke--Imamo\u{g}lu~\cite{article:duke2006} (Prop.~3.1) defines
\begin{align}\label{E: Duke P def}
P_{k,\ell}(x)=\sum_{j=1}^{\ell}\binom{\ell-1}{j-1}(-1)^{j+k}(j-x)^{k-1}.
\end{align}

\smallskip\noindent\textbf{Quarter-integer slots.} At level~$n$, each row index $k\in\{0,\ldots,n\}$ carries a denominator slot $u_{k}:=(4k+3)/4$ and a numerator slot $v_{k}:=(4(n-k)+1)/4$ (for example $u_{0}=3/4$ and $v_{0}=(4n+1)/4$).

\smallskip\noindent\textbf{Multiple gamma.} Write $\Gamma_{m}$ for Vignéras' multiple gamma function with $\Gamma_{1}=\Gamma$ (full normalisation and shift law in Section~\ref{sec:prelims}, paragraph~\ref{par:gamma-convention}).

\smallskip\noindent\textbf{Finite gamma template.} The \emph{finite gamma template}~$\mathcal{S}_{n}$ is the $T$-weighted sum of $\log\Gamma_{n+1}$ slot differences
\begin{align}\label{E: S_n def}
\mathcal{S}_{n}
&:=\sum_{k=0}^{n}T(n,k)\Big[\log\Gamma_{n+1}(v_{k})-\log\Gamma_{n+1}(u_{k})\Big]
\nonumber\\
&=\sum_{k=0}^{n}T(n,k)\Biggl[
\log\Gamma_{n+1}\!\Bigl(\frac{4(n-k)+1}{4}\Bigr)
-\log\Gamma_{n+1}\!\Bigl(\frac{4k+3}{4}\Bigr)
\Biggr].
\end{align}
Equivalently, $\mathcal{S}_{n}=\log\prod_{k=0}^{n}(\Gamma_{n+1}(v_{k})/\Gamma_{n+1}(u_{k}))^{T(n,k)}$: a \emph{finite} gamma ratio, defined for every $n\ge 1$ with no convergence question; it is not the infinite product~$P_{n}$. The two meet only for odd~$n$, where the shifted Type-$N$ criterion evaluates $P_{n}$ (Section~\ref{sec:typeN}) and Section~\ref{sec:Ln} proves $\mathcal{S}_{n}=\log P_{n}$; for even~$n$ the product $P_{n}$ diverges (Remark~\ref{rem:even-diverge}) while $\mathcal{S}_{n}$ still evaluates $\beta'(-n)+(\log 4)\,\beta(-n)$ (Theorem~\ref{thm:main-finite}).

\smallskip\noindent\textbf{Duke side.} Define
\begin{align}\label{E: Duke beta}
\mathcal{D}_{n}:=4^{n}\sum_{\ell=1}^{n+1}\Big[P_{n+1,\ell}\!\Big(\tfrac{1}{4}\Big)\log\Gamma_{\ell}\!\Big(\tfrac{1}{4}\Big)-P_{n+1,\ell}\!\Big(\tfrac{3}{4}\Big)\log\Gamma_{\ell}\!\Big(\tfrac{3}{4}\Big)\Big].
\end{align}

\smallskip\noindent\textbf{Wallis--Eulerian product.} For $k\ge 0$ write the block ratio
\begin{align}\label{E: block def}
B_{n,k}:=\prod_{j=0}^{n}\frac{(4k+4j+3)^{T(n,j)}}{(4k+4(n-j)+1)^{T(n,j)}},
\qquad
U_{n}(k):=\log B_{n,k},
\end{align}
and the infinite product
\begin{align}\label{E: Pn def}
P_{n} := \prod_{k=0}^{\infty}B_{n,k}^{\binom{n+k}{n}}
=\prod_{k=0}^{\infty}\left[ \frac{(4k+3)^{T(n,0)}(4k+7)^{T(n,1)}\cdots(4k+4n+3)^{T(n,n)}}{(4k+1)^{T(n,n)}(4k+5)^{T(n,n-1)}\cdots(4k+4n+1)^{T(n,0)}}\right]^{\binom{n+k}{n}}.
\end{align}
When $P_{n}$ converges (proved for odd~$n$ in Section~\ref{sec:typeN}), we write $\mathcal{L}_{n}:=\log P_{n}$.

\begin{theorem}[Finite gamma template, all $n$]\label{thm:main-finite}
For every $n\ge 1$,
\begin{align}\label{E: finite main}
\mathcal{S}_{n}=\mathcal{D}_{n}=\beta'(-n)+(\log 4)\,\beta(-n).
\end{align}
\end{theorem}

\smallskip\noindent\textit{Proof deferred to Sections~\ref{sec:duke-side} and~\ref{sec:closure}.}
Theorem~\ref{thm:induction} and Proposition~\ref{prop:coeff-expansion} give $\mathcal{S}_{n}=\mathcal{D}_{n}$; Proposition~\ref{prop:duke-side} (via Lemma~\ref{lem:duke-L}) identifies $\mathcal{D}_{n}$ with $\beta'(-n)+(\log 4)\beta(-n)$.

\begin{theorem}[Odd-$n$ product]\label{thm:main-odd}
If $n=2m-1$ is odd, then $P_{n}$ converges absolutely and
\begin{align}\label{E: odd beta product}
e^{\beta'(-n)}=P_{n}.
\end{align}
\end{theorem}

\smallskip\noindent\textit{Proof deferred to Section~\ref{sec:Ln}.}
For odd~$n$, $\beta(-n)=0$. Proposition~\ref{prop:Ln-odd} gives $\log P_{n}=\mathcal{S}_{n}$; apply Theorem~\ref{thm:main-finite}. There is no even-$n$ analogue of~\eqref{E: odd beta product}: for even~$n$ the product $P_{n}$ diverges (Remark~\ref{rem:even-diverge}), and only the finite identity~\eqref{E: finite main} holds.

\section{Proof map}\label{sec:proof-map}

\noindent The finite identity $\mathcal{S}_{n}=\mathcal{D}_{n}$ is a match of coefficients in the quarter-integer basis $\{\log\Gamma_{\ell}(b/4)\}_{\ell,b}$. For odd~$n$ the infinite product is then identified with~$\mathcal{S}_{n}$ by Type-$N$. The argument proceeds in the following order.

\smallskip\noindent\textbf{Coefficient target.} Expanding~$\mathcal{S}_{n}$ via the functional equation of~$\Gamma_{m}$ produces coefficients
\begin{align}\label{E: E def}
E_{\ell,b}(n):=\sum_{k=0}^{n}T(n,k)\Big[A_{n+1,\,4(n-k)+1}(\ell,b)-A_{n+1,\,4k+3}(\ell,b)\Big],
\end{align}
where $A_{m,r}(\ell,b)$ is the coefficient of $\log\Gamma_{\ell}(b/4)$ in the expansion of $\log\Gamma_{m}(r/4)$ (Lemma~\ref{lem:FE}). The finite identity~\eqref{E: finite main} follows once
\begin{align}\label{E: exponent match}
E_{\ell,1}(n)=4^{n}\,P_{n+1,\ell}\!\Big(\tfrac14\Big),\qquad
E_{\ell,3}(n)=-4^{n}\,P_{n+1,\ell}\!\Big(\tfrac34\Big)
\end{align}
for $1\le\ell\le n+1$, since then $\mathcal{S}_{n}=\mathcal{D}_{n}$ by~\eqref{E: Duke beta}. (By definition $A_{m,r}(\ell,b)=0$ for $\ell>m$, so $E_{\ell,b}(n)=0$ whenever $\ell>n+1$.)

\smallskip\noindent\textbf{Steps.}
\begin{enumerate}
\item \textbf{Expand.} Write every $\log\Gamma_{n+1}(r/4)$ in the quarter-integer basis and cancel rational tails (Lemmas~\ref{lem:FE},~\ref{lem:rational-cancel}; Proposition~\ref{prop:coeff-expansion}).
\item \textbf{Close $A$.} Evaluate the FE coefficients in closed form (Lemma~\ref{lem:A-closed}).
\item \textbf{Binomial $E$.} Substitute to express $E_{\ell,b}(n)$ as a $T$-weighted binomial sum (Proposition~\ref{prop:E-binomial}).
\item \textbf{Match Duke.} Identify those binomials with $P_{n+1,\ell}(b/4)$ (Lemma~\ref{lem:T-binom-P}, Theorem~\ref{thm:induction}); identify $\mathcal{D}_{n}$ with $\beta'(-n)+(\log 4)\beta(-n)$ (Proposition~\ref{prop:duke-side}).
\item \textbf{Odd product.} For odd~$n$, Type-$N$ gives $\log P_{n}=\mathcal{S}_{n}$ (Section~\ref{sec:Ln}).
\end{enumerate}

\smallskip\noindent\textbf{Dependency chain.}
\begin{align*}
\mathcal{S}_{n}
\xrightarrow{\mathrm{FE}}
E_{\ell,b}(n)
\xrightarrow{A\text{-closed}}
T\text{-binomials}
\xrightarrow{\mathrm{Eulerian\text{--}Duke}}
P_{n+1,\ell}(b/4)
\longrightarrow
\mathcal{D}_{n},
\end{align*}
and for odd~$n$, Type-$N$ gives $\log P_{n}=\mathcal{S}_{n}$, hence $P_{n}=e^{\beta'(-n)}$.

\smallskip\noindent\textbf{Section order.} Section~\ref{sec:prelims} collects the $T$-sum and functional-equation toolkits. Sections~\ref{sec:duke-side}--\ref{sec:closure} carry out steps 1--4 of the coefficient match, completing the proof of Theorem~\ref{thm:main-finite}; Section~\ref{sec:typeN} then cashes the $T$-sum identities on the product side, and Section~\ref{sec:Ln} assembles the odd-$n$ product (step~5).

\section{Preliminaries}\label{sec:prelims}
\subsection{$T$-sums and toolkit}

\noindent\textit{Reference section.} $T$-sum identities used throughout and the FE/Duke lemmas used in the coefficient match. Skim on a first pass; return when a specific lemma is invoked.

\subsubsection{$T$-sum identities}

Recall that $T(n,k)$, its recurrence~\eqref{E: T recur}, and the slot pair $u_{k}:=(4k+3)/4$, $v_{k}:=(4(n-k)+1)/4$ were fixed in Section~\ref{sec:main-results}.

\begin{lemma}[$T$-row total]\label{lem:tsum}
For every $n\ge 0$, $\sum_{k=0}^{n}T(n,k)=4^{n}n!$.
\end{lemma}

\begin{proof}
From~\eqref{E: T recur}, $\sum_{k}(4k+3)T(n,k)+\sum_{k}(4(n-k)+1)T(n,k)=\sum_{k}T(n+1,k)$, i.e.\ $4(n+1)\sum_{k}T(n,k)=\sum_{k}T(n+1,k)$. Induct from $T(0,0)=1$.
\end{proof}

\begin{lemma}[Moment vanishing]\label{lem:moment}
With $v_{k},u_{k}$ as above, for $0\le m\le n$,
\begin{align}\label{E: moment zero}
\sum_{k=0}^{n}T(n,k)\big[v_{k}^{\,m}-u_{k}^{\,m}\big]=0,
\end{align}
and hence for any polynomial~$R$ of degree $\le n$,
\begin{align}\label{E: poly moment}
\sum_{k=0}^{n}T(n,k)\big[R(v_{k})-R(u_{k})\big]=0.
\end{align}
\end{lemma}

\begin{proof}
It is equivalent (clearing the factor $4^{m}$) to prove
\begin{align}\label{E: moment int}
\sum_{k=0}^{n}T(n,k)\Big[(4k+3)^{m}-(4(n-k)+1)^{m}\Big]=0
\qquad(0\le m\le n).
\end{align}
Write $a_{k}:=4k+3$, $b_{k}:=4(n-k)+1$, and $S:=4n+4$ (so $a_{k}+b_{k}=S$), and set $D_n(r):=\sum_{k}T(n,k)\big(a_{k}^{r}-b_{k}^{r}\big)$. We prove~\eqref{E: moment int} by induction on~$n$.

\smallskip\noindent\textit{Base $n=0$.} Then $T(0,0)=1$ and $m=0$, so $D_0(0)=0$.

\smallskip\noindent\textit{Induction step.} Assume $D_n(r)=0$ for all $0\le r\le n$. At level~$n+1$ the complementary slot is $4(n-k)+5$. From~\eqref{E: T recur}, reindexing $k\mapsto k+1$ in the $T(n,k-1)$ sum yields
\begin{align}\label{E: moment expand A}
\sum_{k=0}^{n+1}T(n+1,k)\,(4k+3)^{m}
&=\sum_{j=0}^{n}T(n,j)\,b_{j}\,(a_{j}+4)^{m}
+\sum_{j=0}^{n}T(n,j)\,a_{j}^{m+1},
\\
\label{E: moment expand B}
\sum_{k=0}^{n+1}T(n+1,k)\,(4(n-k)+5)^{m}
&=\sum_{j=0}^{n}T(n,j)\,b_{j}^{m+1}
+\sum_{j=0}^{n}T(n,j)\,a_{j}\,(b_{j}+4)^{m}.
\end{align}
(The factor $b_{j}=4(n-j)+1$ in~\eqref{E: moment expand A} is forced by the reindex.) Expand $(a_{j}+4)^{m}$ and $(b_{j}+4)^{m}$ binomially and rewrite $b_{j}=S-a_{j}$, $a_{j}=S-b_{j}$. The difference of~\eqref{E: moment expand A} and~\eqref{E: moment expand B} collapses to
\begin{align}\label{E: moment diff formula}
&\sum_{k=0}^{n+1}T(n+1,k)\Big[(4k+3)^{m}-(4(n-k)+5)^{m}\Big]
\nonumber\\
&\qquad=\sum_{q=0}^{m}\Big(S\binom{m}{q}4^{m-q}-\binom{m}{q-1}4^{m-q+1}\Big)D_n(q),
\end{align}
where $\binom{m}{-1}:=0$. (No $D_n(m+1)$ term appears: the reindexed sum contributes it with coefficient $-\binom{m}{m}4^{m-(m+1)+1}=-1$, cancelling the $+1$ from the $a_{j}^{m+1}-b_{j}^{m+1}$ terms.) For $m\le n$ every order $q\le m$ vanishes by the inductive hypothesis. For $m=n+1$ the orders $q\le n$ vanish by hypothesis, and the surviving top order $q=m$ carries coefficient $S\binom{m}{m}4^{0}-\binom{m}{m-1}4^{1}=S-4m=(4n+4)-4(n+1)=0$. Hence the right side vanishes for every $0\le m\le n+1$, so the left side of~\eqref{E: moment diff formula} is zero. This is~\eqref{E: moment int} at level~$n+1$.

\smallskip\noindent The polynomial form~\eqref{E: poly moment} follows by linearity.
\end{proof}

\begin{lemma}[Low power sums]\label{lem:low-power}
For every $n\ge 1$ and $1\le j\le n$,
\begin{align}\label{E: low power}
\sum_{k=0}^{n}(4k+3)^{j}\,T(n,k)=\sum_{k=0}^{n}(4(n-k)+1)^{j}\,T(n,k).
\end{align}
\end{lemma}

\begin{proof}
Multiply~\eqref{E: moment zero} by $4^{j}$: $(4k+3)^{j}=4^{j}u_{k}^{\,j}$ and $(4(n-k)+1)^{j}=4^{j}v_{k}^{\,j}$.
\end{proof}

\begin{lemma}[$\beta$ at negative integers]\label{lem:beta-neg}
For every $n\ge 0$, with $E_{n}$ the $n$th Euler number in the DLMF normalisation $E_{0}=1$, $E_{2}=-1$, $E_{4}=5$, \ldots,
\begin{align}\label{E: beta neg values}
\beta(-n)=\begin{cases}
0 & \text{if $n$ is odd},\\[2pt]
E_{n}/2 & \text{if $n$ is even},
\end{cases}
\end{align}
(standard; see e.g.\ Weisstein~\cite{weisstein2002beta}, or equivalently the Hurwitz form
\begin{align}\label{E: beta hurwitz}
\beta(-n)=\frac{4^{n}}{n+1}\Bigl(B_{n+1}\Bigl(\tfrac{3}{4}\Bigr)-B_{n+1}\Bigl(\tfrac{1}{4}\Bigr)\Bigr)
\end{align}
via $\beta(s)=4^{-s}\bigl(\zeta(s,1/4)-\zeta(s,3/4)\bigr)$ and $\zeta(-n,a)=-B_{n+1}(a)/(n+1)$~\cite[Thm.~12.13]{apostol1998analytic}, together with DLMF~\cite[(24.4.28), (24.4.31)]{dlmf2010}).
\end{lemma}

\begin{proof}
Classical, as cited. (Odd-$n$ vanishing of the top $T$-moment is Lemma~\ref{lem:top-power-odd}; the even top-moment identity used only in Remark~\ref{rem:even-diverge} is sketched there.)
\end{proof}

\paragraph{Sketch (even top-power defect).}\label{lem:top-power-defect}
For even~$n\ge 2$, the top $T$-moment is expected to satisfy
\begin{align}\label{E: top power defect}
\sum_{k=0}^{n}T(n,k)\Big[(4k+3)^{n+1}-(4(n-k)+1)^{n+1}\Big]
=4^{n+1}(n+1)!\,\beta(-n)
\end{align}
(with both sides zero when $n$ is odd, by Lemma~\ref{lem:top-power-odd} and~\eqref{E: beta neg values}). Outline: reduce the left side via Lemma~\ref{lem:moment} to a Bernoulli difference, insert the inclusion-exclusion formulae for~$T(n,k)$, and evaluate the resulting alternating sum against~\eqref{E: beta hurwitz}. We do not spell out that calculation; the identity is used only to sketch even-$n$ divergence below, where $\beta(-n)\neq 0$ is the essential input.

\begin{lemma}[Top power sum, odd $n$]\label{lem:top-power-odd}
If $n=2m-1$ is odd, then
\begin{align}\label{E: top power odd}
\sum_{k=0}^{n}(4k+3)^{n+1}\,T(n,k)=\sum_{k=0}^{n}(4(n-k)+1)^{n+1}\,T(n,k).
\end{align}
\end{lemma}

\begin{proof}
Write $a_{k}:=4k+3$, $b_{k}:=4(n-k)+1$, and $S':=4(n+1)$, so $a_{k}+b_{k}=S'$. The claim is $\sum_{k}T(n,k)\,(a_{k}^{n+1}-b_{k}^{n+1})=0$. Expand
\[
a_{k}^{n+1}-b_{k}^{n+1}=(S'-b_{k})^{n+1}-b_{k}^{n+1}
=\sum_{r=0}^{n}\binom{n+1}{r}(S')^{n+1-r}(-1)^{r}b_{k}^{r}
+\bigl((-1)^{n+1}-1\bigr)b_{k}^{n+1}.
\]
Since $n$ is odd, $n+1$ is even and $(-1)^{n+1}=1$, so the degree-$n+1$ terms cancel and
\begin{align}\label{E: top power poly}
a_{k}^{n+1}-b_{k}^{n+1}
=\sum_{r=0}^{n}c_{r}\,b_{k}^{r},
\qquad
c_{r}:=\binom{n+1}{r}(S')^{n+1-r}(-1)^{r}.
\end{align}
The same expansion with the roles of~$a_{k}$ and~$b_{k}$ reversed yields $b_{k}^{n+1}-a_{k}^{n+1}=\sum_{r=0}^{n}c_{r}\,a_{k}^{r}$, hence $a_{k}^{n+1}-b_{k}^{n+1}=-\sum_{r=0}^{n}c_{r}\,a_{k}^{r}$. Therefore
\begin{align*}
\sum_{k}T(n,k)\,(a_{k}^{n+1}-b_{k}^{n+1})
&=\sum_{r=0}^{n}c_{r}\sum_{k}T(n,k)\,b_{k}^{r}
=-\sum_{r=0}^{n}c_{r}\sum_{k}T(n,k)\,a_{k}^{r}.
\end{align*}
For $1\le r\le n$, Lemma~\ref{lem:low-power} gives $\sum_{k}T(n,k)\,a_{k}^{r}=\sum_{k}T(n,k)\,b_{k}^{r}$; for $r=0$ both sides equal $\sum_{k}T(n,k)$. The two expressions for the top-moment difference are therefore negatives of each other, so the difference vanishes.
\end{proof}

\begin{lemma}[Top power at level $n+1$, all $n$]\label{lem:top-power-shift}
For every $n\ge 0$ and $0\le m\le n+1$,
\begin{align}\label{E: top power shift}
\sum_{k=0}^{n+1}(4k+3)^{m}\,T(n+1,k)=\sum_{k=0}^{n+1}(4(n-k)+5)^{m}\,T(n+1,k).
\end{align}
\end{lemma}

\begin{proof}
This is~\eqref{E: moment int} at level~$N=n+1$ (where $4((n+1)-k)+1=4(n-k)+5$), already established in the inductive step of Lemma~\ref{lem:moment}.
\end{proof}

\smallskip\noindent\textbf{Type-$N$ input from Part~I.} The product side below rests on Theorem~\ref{thm:typeN}: part~(i) is the Type-$N$ product with exponent $\binom{n+k-2}{n-1}$, and part~(ii) is the shifted criterion with exponent $\binom{n+k}{n}$ and value $\prod_{\ell}\Gamma_{n+1}(b_{\ell})/\Gamma_{n+1}(a_{\ell})$, whose power-sum hypotheses are exactly the $T$-sum identities established above.

\subsubsection{Multiple-gamma and Duke toolkit}

\smallskip\noindent To match $\mathcal{S}_{n}$ to $\mathcal{D}_{n}$ we expand every $\log\Gamma_{n+1}(r/4)$ into the quarter-integer basis; the coefficient target is~\eqref{E: E def}--\eqref{E: exponent match}.

\smallskip\noindent\textbf{Multiple-gamma shift.} For $m\ge 1$, $z\notin\mathbb{Z}_{\leq 0}$ (Choi~\cite{ARTICLE:2}),
\begin{align}\label{E: FE shift}
\log\Gamma_{m}(z)=\log\Gamma_{m}(z-1)-\log\Gamma_{m-1}(z-1).
\end{align}
At $m=1$ read $\Gamma_{0}(z):=1/z$, so~\eqref{E: FE shift} is the classical shift $\log\Gamma(z)=\log\Gamma(z-1)+\log(z-1)$.
The next lemma expands each $\log\Gamma_{m}(r/4)$ at odd~$r$ into the quarter-integer basis $\log\Gamma_{\ell}(b/4)$ with $b\in\{1,3\}$; the same odd integer~$r$ indexes the argument $r/4$, and at the \emph{pivot} values $r\in\{1,3\}$ one has $r/4=b/4$ with $b=r$. Write $A_{m,r}(\ell,b)$ for the coefficient of $\log\Gamma_{\ell}(b/4)$ in this expansion.

\begin{lemma}[Quarter-integer expansion]\label{lem:FE}
Fix $m\ge 1$ and odd~$r$ (so the argument is $z=r/4\in\{1/4,3/4,5/4,\ldots\}$). For $1\le\ell\le m$ and $b\in\{1,3\}$, define integers $A_{m,r}(\ell,b)$ recursively as follows.
\begin{enumerate}
\item[(i)] \textbf{Pivot base ($r=b\in\{1,3\}$).} At the basis points $r/4=b/4\in\{1/4,3/4\}$ the expansion is the identity: $A_{m,r}(\ell,b)=1$ when $\ell=m$ and $r=b$, and $A_{m,r}(\ell,b)=0$ otherwise; explicitly $A_{m,1}(m,1)=1$ ($r=1$, argument $1/4$) and $A_{m,3}(m,3)=1$ ($r=3$, argument $3/4$).
\item[(ii)] \textbf{Shift step (odd $r\ge 5$).} For arguments $r/4$ above the pivots, descend by four in the numerator: $A_{m,r}(\ell,b)=A_{m,r-4}(\ell,b)-A_{m-1,r-4}(\ell,b)$.
\end{enumerate}
Then
\begin{align*}
\log\Gamma_{m}\!\Big(\frac{r}{4}\Big)=\sum_{\ell=1}^{m}\sum_{b\in\{1,3\}}A_{m,r}(\ell,b)\,\log\Gamma_{\ell}\!\Big(\frac{b}{4}\Big)
+R_{m,r},
\end{align*}
where $R_{m,r}$ is a \emph{rational linear combination} of $\log 2$ and finitely many terms $\log q$ with $q$ a positive odd integer and $q\le r$ (not every such~$q$ need appear).
\end{lemma}

\begin{proof}
Iterate~\eqref{E: FE shift} at $z=r/4$, using $\log\Gamma_{1}(z)=\log\Gamma(z)$ and, when $m=1$, the classical shift $\log\Gamma(z)=\log\Gamma(z-1)+\log(z-1)$ (Whittaker--Watson~\cite[Ex.~6, \S12.1]{article:8}). Since $r$ is odd, every argument met in the descent is a quarter-integer with odd numerator, so each $m=1$ step at argument $z=r'/4$ contributes $\log(z-1)=\log((r'-4)/4)=\log(r'-4)-2\log 2$ with $r'-4$ a positive odd integer, and no other scalar terms arise. Collecting these yields~$R_{m,r}$.
\end{proof}

\begin{lemma}[Rational tail cancellation]\label{lem:rational-cancel}
With $R_{m,r}$ as in Lemma~\ref{lem:FE}, for every $n\ge 1$,
\begin{align}\label{E: rational cancel}
\sum_{k=0}^{n}T(n,k)\Big[R_{n+1,\,4(n-k)+1}-R_{n+1,\,4k+3}\Big]=0.
\end{align}
Hence $\mathcal{S}_{n}$ expands purely in the quarter-integer basis~\eqref{E: L expanded}.
\end{lemma}

\begin{proof}
Insert Lemma~\ref{lem:FE} into~\eqref{E: S_n def}. The $\log\Gamma_{\ell}(b/4)$ part is~\eqref{E: E def}; it remains to show~\eqref{E: rational cancel}.

The elementary remainders obey the same shift as the multiple gamma:
\begin{align}\label{E: R recur}
R_{m,r}=R_{m,r-4}-R_{m-1,r-4}\qquad(m\ge 2,\; r\ge 5),
\end{align}
with $R_{m,b}=0$ for $b\in\{1,3\}$ and, at level~$m=1$,
\begin{align}\label{E: R ordinary}
R_{1,r}
&=\sum_{j=0}^{t-1}\log\Bigl(\frac{b+4j}{4}\Bigr)
=\sum_{j=0}^{t-1}\log(b+4j)\;-\;2t\log 2,
\nonumber\\
&\qquad t=\frac{r-b}{4},\; b\equiv r\pmod{4},\; b\in\{1,3\}.
\end{align}
Write $\rho_{m}(r;\tau)$ for the coefficient of a basis element $\tau\in\{\log 2\}\cup\{\log q:q\ge 3\text{ odd}\}$ in $R_{m,r}$. If $r=b+4t$ with $b\in\{1,3\}$ and $t\ge 0$, then
\begin{align}\label{E: R closed}
\rho_{m}(b+4t;\log 2)&=(-1)^{m}\,2\binom{t}{m},
\qquad
\rho_{m}\big(b+4t;\log(b+4j)\big)=(-1)^{m+1}\binom{t-j-1}{m-1}
\end{align}
for $0\le j\le t-1$ with $b+4j\ge 3$ (binomial coefficients vanish when the upper index is negative or the lower exceeds the upper). Both formulae follow by induction on~$m$ from~\eqref{E: R recur}--\eqref{E: R ordinary}: the base $m=1$ is~\eqref{E: R ordinary}, and the Pascal identities $\binom{t}{m}=\binom{t-1}{m}+\binom{t-1}{m-1}$ and $\binom{t-j-1}{m-1}=\binom{t-j-2}{m-1}+\binom{t-j-2}{m-2}$ match the subtraction step.

In~\eqref{E: rational cancel} the numerator slot $4(n-k)+1$ has residue~$1$ and height $t=n-k\le n$, while the denominator slot $4k+3$ has residue~$3$ and height $t=k\le n$. At level $m=n+1$, every binomial in~\eqref{E: R closed} has lower index $n+1$ or~$n$, hence vanishes for all heights $t\le n$. Thus $\rho_{n+1}(r;\tau)=0$ for every template slot~$r$ and every~$\tau$, which is~\eqref{E: rational cancel}.
\end{proof}

\begin{lemma}[Closed form for FE coefficients]\label{lem:A-closed}
For $m\ge 1$, $b\in\{1,3\}$, $1\le\ell\le m$, and $t\ge 0$,
\begin{align}\label{E: A closed}
A_{m,\,b+4t}(\ell,b)=(-1)^{m-\ell}\binom{t}{m-\ell},
\qquad
A_{m,\,r}(\ell,b)=0\text{ whenever }r\not\equiv b\pmod{4}.
\end{align}
(The binomial coefficient is~$0$ when $t<m-\ell$.)
\end{lemma}

\begin{proof}
Both claims follow by induction on~$m$ from Lemma~\ref{lem:FE}. At a matching pivot $r=b$ one has $t=0$, so~\eqref{E: A closed} reduces to $\delta_{m\ell}$. At a mismatched pivot $r\in\{1,3\}\setminus\{b\}$ the value is~$0$. The shift step $A_{m,r}=A_{m,r-4}-A_{m-1,r-4}$ preserves the vanishing for $r\not\equiv b\pmod 4$, and on the matching residue class the Pascal identity $\binom{t}{m-\ell}=\binom{t-1}{m-\ell}+\binom{t-1}{m-\ell-1}$ matches the subtraction.
\end{proof}

\paragraph{Convention (multiple gamma).}\label{par:gamma-convention}
Throughout, $\Gamma_{m}$ denotes Vignéras' multiple gamma function with $\Gamma_{1}=\Gamma$ and shift~\eqref{E: FE shift}. Duke--Imamo\u{g}lu~\cite{article:duke2006} work in the same normalisation; their identity~\eqref{E: adam hurwitz} (recalled in the proof of Lemma~\ref{lem:duke-L}) converts Hurwitz zeta derivatives into $\log\Gamma_{\ell}$. Barnes' $G_{N}$ is related by $\log\Gamma_{N}=(-1)^{N}\log G_{N}$ when $G_{1}=1/\Gamma$, but we never need~$G_{N}$ below.

\begin{lemma}[Duke polynomial]\label{lem:duke-P}
The polynomials $P_{k,\ell}(x)$ are as in~\eqref{E: Duke P def} (Duke--Imamo\u{g}lu~\cite{article:duke2006}, Prop.~3.1). In particular $P_{2,1}(x)=x-1$ and $P_{2,2}(x)=1$.
\end{lemma}

\begin{proof}
Duke--Imamo\u{g}lu~\cite{article:duke2006}, Prop.~3.1; the $k=2$ values are immediate from~\eqref{E: Duke P def}.
\end{proof}

\paragraph{What the preliminaries buy us.}\label{par:prelims-narrative}
The $T$-sum layer (Lemmas~\ref{lem:moment}--\ref{lem:top-power-shift}) is what makes \emph{lower powers cancel}: weighted slot moments vanish because the Eulerian recurrence~\eqref{E: T recur} acts like discrete integration by parts on monomials in the slot variables. Type-$N$ (Theorem~\ref{thm:typeN}) packages the odd-$n$ infinite product as a \emph{finite} gamma ratio; this is how the product is identified with the template in Section~\ref{sec:typeN}. The FE toolkit (Lemmas~\ref{lem:FE}--\ref{lem:A-closed}) expands $\mathcal{S}_{n}$ into the quarter-integer basis and evaluates the coefficients~$A_{m,r}$; the binomial form of~$E_{\ell,b}$ and the match to Duke's~$P_{n+1,\ell}$ are deferred to Section~\ref{sec:closure}.

\section{Duke side and the coefficient target}\label{sec:duke-side}

\noindent\textbf{Goal.} Identify $\mathcal{D}_{n}$ with $\beta'(-n)+(\log 4)\beta(-n)$ and expand $\mathcal{S}_{n}$ in the quarter-integer basis using the already-stated target~\eqref{E: E def}--\eqref{E: exponent match}. The base case $n=1$ is verified explicitly.

\begin{lemma}[Duke--Imamo\u{g}lu $L$-function formula]\label{lem:duke-L}
Let $\chi$ be a nontrivial Dirichlet character modulo~$N$ and let $k\in\mathbb{Z}_{+}$. With the polynomials $P_{k,\ell}$ of~\eqref{E: Duke P def},
\begin{align}\label{E: duke L general}
L'(1-k,\chi)+(\log N)\,L(1-k,\chi)
=N^{k-1}\sum_{\ell=1}^{k}\sum_{r=1}^{N}\chi(r)\,P_{k,\ell}\!\Big(\frac{r}{N}\Big)\,\log\Gamma_{\ell}\!\Big(\frac{r}{N}\Big).
\end{align}
This is the display immediately following~(11) in Duke--Imamo\u{g}lu~\cite[\S3]{article:duke2006} (their~(11) is~\eqref{E: Duke P def}).
\end{lemma}

\begin{proof}
Page-check against Duke--Imamo\u{g}lu~\cite[\S3]{article:duke2006}: their~(10) is Adamchik's identity
\begin{align}\label{E: adam hurwitz}
\zeta'(1-k,x)-\zeta'(1-k)=\sum_{\ell=1}^{k}P_{k,\ell}(x)\,\log\Gamma_{\ell}(x)
\end{align}
($k\ge 1$, $\Re(x)>0$), and their~(11) is~\eqref{E: Duke P def}. For nontrivial $\chi\bmod N$,
\[
L(s,\chi)=N^{-s}\sum_{r=1}^{N}\chi(r)\,\zeta(s,r/N)
\]
(Apostol~\cite[Ch.~12]{apostol1998analytic}). Differentiating at $s=1-k$: writing $L(1-k,\chi)=N^{k-1}\sum_r\chi(r)\zeta(1-k,r/N)$, one obtains
\[
L'(1-k,\chi)+(\log N)\,L(1-k,\chi)
=N^{k-1}\sum_{r=1}^{N}\chi(r)\,\zeta'(1-k,r/N).
\]
(The constant terms $\zeta'(1-k)$ cancel because $\sum_r\chi(r)=0$.) Insert~\eqref{E: adam hurwitz} to arrive at~\eqref{E: duke L general}, matching the published display. Hypotheses: $k\in\mathbb{Z}_{+}$, $\chi$ nontrivial (so $L(1-k,\chi)$ is finite), and $\Re(r/N)>0$ whenever $\chi(r)\ne 0$.
\end{proof}

\begin{proposition}[Duke side]\label{prop:duke-side}
For the primitive character $\chi$ modulo~$4$,
\begin{align}\label{E: duke beta id}
\beta'(-n)+(\log 4)\,\beta(-n)=\mathcal{D}_{n}
\end{align}
for $\mathcal{D}_{n}$ as in~\eqref{E: Duke beta}.
\end{proposition}

\begin{proof}
Take $N=4$, $k=n+1$, and $\chi$ the non-principal character modulo~$4$, so $\chi(1)=1$, $\chi(3)=\chi(-1)=-1$, and $\chi(2)=\chi(4)=0$ (Apostol~\cite[Ch.~8]{apostol1998analytic}). Then $L(1-k,\chi)=L(-n,\chi)=\beta(-n)$ and $L'(1-k,\chi)=\beta'(-n)$ (DLMF~\cite[\S4.19]{dlmf2010}). Only $r=1$ and $r=3$ contribute in~\eqref{E: duke L general}, giving
\[
\beta'(-n)+(\log 4)\beta(-n)
=4^{n}\sum_{\ell=1}^{n+1}\Big[P_{n+1,\ell}\!\Big(\tfrac14\Big)\log\Gamma_{\ell}\!\Big(\tfrac14\Big)
-P_{n+1,\ell}\!\Big(\tfrac34\Big)\log\Gamma_{\ell}\!\Big(\tfrac34\Big)\Big]
=\mathcal{D}_{n}
\]
by~\eqref{E: Duke beta}.
\end{proof}

\smallskip\noindent\textbf{Sign convention.} In the expansion of~$\mathcal{S}_n$, the coefficient at $b=1$ is $E_{\ell,1}(n)=4^{n}P_{n+1,\ell}(1/4)$ and at $b=3$ is $E_{\ell,3}(n)=-4^{n}P_{n+1,\ell}(3/4)$, matching the $\pm$ structure of~\eqref{E: Duke beta}. When $n$ is even, $\beta(-n)\neq 0$ and the full exponent $e^{\beta'(-n)+(\log 4)\beta(-n)}$ is required.

\begin{proposition}[Coefficient expansion]\label{prop:coeff-expansion}
Apply Lemma~\ref{lem:FE} to~\eqref{E: S_n def}. With $E_{\ell,b}(n)$ as in~\eqref{E: E def}, the template expands as
\begin{align}\label{E: L expanded}
\mathcal{S}_{n}=\sum_{\ell=1}^{n+1}\sum_{b\in\{1,3\}}E_{\ell,b}(n)\,\log\Gamma_{\ell}\!\Big(\frac{b}{4}\Big),
\end{align}
because rational $\log 2$ and $\log p$ ($p$ odd) terms cancel in the $T$-weighted sum (Lemma~\ref{lem:rational-cancel}). Once the coefficient target~\eqref{E: exponent match} holds, $\mathcal{S}_{n}=\mathcal{D}_{n}$ by~\eqref{E: Duke beta}.
\end{proposition}

\begin{proof}
Expand~\eqref{E: S_n def} with Lemma~\ref{lem:FE}; collect $A_{n+1,r}(\ell,b)$ into~\eqref{E: E def}. Rational tails vanish by Lemma~\ref{lem:rational-cancel}.
\end{proof}

\begin{proposition}[Base case $n=1$]\label{prop:base-case}
At $n=1$, direct expansion with Lemma~\ref{lem:FE} gives $E_{1,1}(1)=-3$, $E_{1,3}(1)=1$, $E_{2,1}(1)=4$, $E_{2,3}(1)=-4$, matching $4P_{2,\ell}(b/4)$ from Lemma~\ref{lem:duke-P}.
\end{proposition}

\begin{proof}
From~\eqref{E: T recur} with $T(0,0)=1$, $T(1,0)=3$, and $T(1,1)=1$. Thus
\[
\mathcal{S}_{1}=3\Big[\log\Gamma_{2}\!\Big(\tfrac{5}{4}\Big)-\log\Gamma_{2}\!\Big(\tfrac{3}{4}\Big)\Big]
+\Big[\log\Gamma_{2}\!\Big(\tfrac{1}{4}\Big)-\log\Gamma_{2}\!\Big(\tfrac{7}{4}\Big)\Big].
\]
Apply Lemma~\ref{lem:FE} at each odd argument $r\in\{1,3,5,7\}$ and read off the $\log\Gamma_{\ell}(b/4)$ coefficients via~\eqref{E: E def}. The pivot base~(i) and shift step~(ii) give the explicit table
\begin{center}
\small
\begin{tabular}{c|cccc}
$r$ & $A_{2,r}(1,1)$ & $A_{2,r}(1,3)$ & $A_{2,r}(2,1)$ & $A_{2,r}(2,3)$ \\ \hline
$1$ & $0$ & $0$ & $1$ & $0$ \\
$3$ & $0$ & $0$ & $0$ & $1$ \\
$5$ & $-1$ & $0$ & $1$ & $0$ \\
$7$ & $0$ & $-1$ & $0$ & $1$
\end{tabular}
\end{center}
For example, $A_{2,5}(1,1)=A_{2,1}(1,1)-A_{1,1}(1,1)=0-1=-1$ and $A_{2,5}(2,1)=A_{2,1}(2,1)-A_{1,1}(2,1)=1-0=1$; the rows $r=3,7$ are similar. Collecting,
\begin{align*}
E_{1,1}(1)&=3\big(A_{2,5}(1,1)-A_{2,3}(1,1)\big)+\big(A_{2,1}(1,1)-A_{2,7}(1,1)\big)=3(-1)+0=-3,\\
E_{1,3}(1)&=3\big(A_{2,5}(1,3)-A_{2,3}(1,3)\big)+\big(A_{2,1}(1,3)-A_{2,7}(1,3)\big)=0+1=1,\\
E_{2,1}(1)&=3(1-0)+(1-0)=4,\qquad
E_{2,3}(1)=3(0-1)+(0-1)=-4.
\end{align*}
On the Duke side, Lemma~\ref{lem:duke-P} with $k=2$ gives $P_{2,1}(x)=x-1$ and $P_{2,2}(x)=1$, hence
$4P_{2,1}(1/4)=-3$, $-4P_{2,1}(3/4)=1$, $4P_{2,2}(1/4)=4$, $-4P_{2,2}(3/4)=-4$, matching~\eqref{E: exponent match}.
\end{proof}

\section{Coefficient identity}\label{sec:closure}

\paragraph{Narrative: closing the loop.}\label{par:closure-narrative}
It remains only to identify the $T$-binomial form of $E_{\ell,b}(n)$ with Duke's $P_{n+1,\ell}$. Proposition~\ref{prop:coeff-expansion} expands $\mathcal{S}_{n}$ as $\sum E_{\ell,b}(n)\log\Gamma_{\ell}(b/4)$; Lemma~\ref{lem:A-closed} and Proposition~\ref{prop:E-binomial} evaluate those coefficients as $T$-weighted binomials; Lemma~\ref{lem:T-binom-P} identifies the binomials with Duke's $P_{n+1,\ell}(b/4)$. Theorem~\ref{thm:induction} records the match. Sections~\ref{sec:typeN}--\ref{sec:Ln} then identify the odd infinite product with~$\mathcal{S}_{n}$.

\noindent\textbf{Goal.} Close the coefficient target~\eqref{E: exponent match}.

\begin{proposition}[Binomial form of $E_{\ell,b}$]\label{prop:E-binomial}
For every $n\ge 1$ and $1\le\ell\le n+1$,
\begin{align}\label{E: E binomial}
E_{\ell,1}(n)&=(-1)^{n+1-\ell}\sum_{k=0}^{n}T(n,k)\binom{n-k}{n+1-\ell},
\\
E_{\ell,3}(n)&=(-1)^{n-\ell}\sum_{k=0}^{n}T(n,k)\binom{k}{n+1-\ell}.
\end{align}
\end{proposition}

\begin{proof}
Apply Lemma~\ref{lem:A-closed} in~\eqref{E: E def}. For $b=1$ the denominator slots $4k+3$ are incongruent to~$1$, so their $A$-coefficients vanish, while the numerator slot $4(n-k)+1$ has height $t=n-k$. For $b=3$ the roles reverse (height $t=k$ on the denominator side), and the overall minus sign in~\eqref{E: E def} produces the displayed factor $(-1)^{n-\ell}$.
\end{proof}

\begin{lemma}[Eulerian--Duke binomial identity]\label{lem:T-binom-P}
For every $n\ge 0$ and $1\le\ell\le n+1$,
\begin{align}\label{E: T binom 1}
\sum_{k=0}^{n}T(n,k)\binom{n-k}{n+1-\ell}
&=\sum_{j=1}^{\ell}\binom{\ell-1}{j-1}(-1)^{j-\ell}(4j-1)^{n},
\\
\label{E: T binom 3}
\sum_{k=0}^{n}T(n,k)\binom{k}{n+1-\ell}
&=\sum_{j=1}^{\ell}\binom{\ell-1}{j-1}(-1)^{j-\ell}(4j-3)^{n}.
\end{align}
Consequently~\eqref{E: exponent match} holds for all $n\ge 1$.
\end{lemma}

\begin{proof}
Expanding~\eqref{E: Duke P def} at $x=\tfrac14$ gives
\[
4^{n}P_{n+1,\ell}\!\Big(\tfrac14\Big)
=\sum_{j=1}^{\ell}\binom{\ell-1}{j-1}(-1)^{j+n+1}(4j-1)^{n},
\]
so $(-1)^{n+1-\ell}4^{n}P_{n+1,\ell}(1/4)$ equals the right side of~\eqref{E: T binom 1}; likewise at $x=\tfrac34$ with $4j-3$. Combined with Proposition~\ref{prop:E-binomial}, the identity~\eqref{E: exponent match} reduces to~\eqref{E: T binom 1}--\eqref{E: T binom 3}.

We prove~\eqref{E: T binom 1} by induction on~$n$, working directly with the two displayed sums: they agree at $n=0$, both vanish whenever $\ell>n+1$, and both satisfy the same recurrence taking level~$n$ to level~$n+1$. Throughout, $\binom{x}{m}:=0$ for integer $m<0$.

\smallskip\noindent\textit{Boundary.} At $n=0$, $\ell=1$, both sides of~\eqref{E: T binom 1} equal~$1$. For $\ell>n+1$ both sides vanish, at every level: on the left every $\binom{n-k}{n+1-\ell}$ has negative lower index; on the right, substituting $j=i+1$ turns the sum into $(-1)^{1-\ell}\sum_{i=0}^{\ell-1}\binom{\ell-1}{i}(-1)^{i}(4i+3)^{n}$, which up to sign is the $(\ell-1)$-st finite difference of the polynomial $(4x+3)^{n}$ at $x=0$, and a finite difference of order exceeding the degree annihilates a polynomial (Euler~\cite{euler1755institutiones,gould1978euler}).

\smallskip\noindent\textit{Right side.} Write $(4j-1)^{n+1}=(4j-1)(4j-1)^{n}$ and split the extra factor against the target index, $4j-1=(4\ell-1)-4(\ell-j)$:
\begin{align*}
\sum_{j=1}^{\ell}\binom{\ell-1}{j-1}(-1)^{j-\ell}(4j-1)^{n+1}
&=(4\ell-1)\sum_{j=1}^{\ell}\binom{\ell-1}{j-1}(-1)^{j-\ell}(4j-1)^{n}
\\
&\qquad-4\sum_{j=1}^{\ell}(\ell-j)\binom{\ell-1}{j-1}(-1)^{j-\ell}(4j-1)^{n}.
\end{align*}
Now $(\ell-j)\binom{\ell-1}{j-1}=(\ell-1)\binom{\ell-2}{j-1}$: for $j<\ell$ both sides equal $(\ell-1)!/\big((j-1)!\,(\ell-1-j)!\big)$, and at $j=\ell$ both vanish. Since $(-1)^{j-\ell}=-(-1)^{j-(\ell-1)}$, the subtracted sum is $-(\ell-1)$ times the right side of~\eqref{E: T binom 1} at $(n,\ell-1)$. Hence the right side at $(n+1,\ell)$ equals $(4\ell-1)$ times the right side at $(n,\ell)$ plus $4(\ell-1)$ times the right side at $(n,\ell-1)$.

\smallskip\noindent\textit{Left side.} The same recurrence follows from the Eulerian recurrence~\eqref{E: T recur}. Insert~\eqref{E: T recur} and shift $k\mapsto k+1$ in the $T(n,k-1)$ sum:
\begin{align*}
\sum_{k=0}^{n+1}T(n+1,k)\binom{n+1-k}{n+2-\ell}
&=\sum_{k=0}^{n}\big(4(n-k)+1\big)\,T(n,k)\binom{n-k}{n+2-\ell}
\\
&\qquad+\sum_{k=0}^{n}(4k+3)\,T(n,k)\binom{n+1-k}{n+2-\ell}.
\end{align*}
Next, Pascal's rule $\binom{n+1-k}{n+2-\ell}=\binom{n-k}{n+2-\ell}+\binom{n-k}{n+1-\ell}$ in the second sum lets the two slot weights recombine into their constant total $\big(4(n-k)+1\big)+(4k+3)=4n+4$ on the $\binom{n-k}{n+2-\ell}$ terms:
\begin{align*}
\sum_{k=0}^{n+1}T(n+1,k)\binom{n+1-k}{n+2-\ell}
&=(4n+4)\sum_{k=0}^{n}T(n,k)\binom{n-k}{n+2-\ell}
\\
&\qquad+\sum_{k=0}^{n}(4k+3)\,T(n,k)\binom{n-k}{n+1-\ell}.
\end{align*}
The first sum on the right is the left side of~\eqref{E: T binom 1} at $(n,\ell-1)$, since $n+2-\ell=n+1-(\ell-1)$. In the remaining sum, write $4k+3=(4n+4)-\big(4(n-k)+1\big)$ and absorb the resulting factor $n-k$ into the binomial coefficient via
\[
(n-k)\binom{n-k}{n+1-\ell}
=(n+1-\ell)\binom{n-k}{n+1-\ell}+(n+2-\ell)\binom{n-k}{n+2-\ell}
\]
(the identity $x\binom{x}{m}=m\binom{x}{m}+(m+1)\binom{x}{m+1}$ at $x=n-k$, $m=n+1-\ell$; both sides are $x!/\big(m!\,(x-m-1)!\big)$ when $x>m\ge 0$, and the edge cases vanish or agree trivially). This evaluates the remaining sum as
\begin{align*}
\Big[(4n+4)-\big(4(n+1-\ell)+1\big)\Big]&\sum_{k=0}^{n}T(n,k)\binom{n-k}{n+1-\ell}
\\
&\qquad-4(n+2-\ell)\sum_{k=0}^{n}T(n,k)\binom{n-k}{n+2-\ell},
\end{align*}
and $(4n+4)-\big(4(n+1-\ell)+1\big)=4\ell-1$. Collecting, the $\binom{n-k}{n+2-\ell}$ sums carry total coefficient $(4n+4)-4(n+2-\ell)=4(\ell-1)$, so the left side at $(n+1,\ell)$ equals $(4\ell-1)$ times the left side at $(n,\ell)$ plus $4(\ell-1)$ times the left side at $(n,\ell-1)$, which is the same recurrence as the right side.

\smallskip\noindent\textit{Induction.} At $\ell=1$ the $(n,\ell-1)$ terms carry coefficient $4(\ell-1)=0$, and at $\ell=n+2$ the $(n,\ell)$ instances are equal (both zero) by the boundary paragraph; all other instances are equal by the induction hypothesis. Starting from $n=0$, the common recurrence therefore propagates equality to every level, proving~\eqref{E: T binom 1}.

\smallskip\noindent For~\eqref{E: T binom 3}, run the same four moves with the two slot families exchanged. On the right, split $4j-3=(4\ell-3)-4(\ell-j)$ and apply the same identity $(\ell-j)\binom{\ell-1}{j-1}=(\ell-1)\binom{\ell-2}{j-1}$. On the left, the reindexed sum now carries $\binom{k+1}{n+2-\ell}$; Pascal's rule again recombines the slot weights into $4n+4$, and the complement $4(n-k)+1=(4n+4)-(4k+3)$ followed by the absorption $k\binom{k}{m}=m\binom{k}{m}+(m+1)\binom{k}{m+1}$ leaves the common recurrence with $4\ell-3$ in place of $4\ell-1$. The boundary checks are identical, with $(4x+1)^{n}$ in the finite-difference argument, and the induction closes as before.
\end{proof}

\begin{theorem}[Coefficient identity]\label{thm:induction}
For all $n\ge 1$, $1\le\ell\le n+1$, and $b\in\{1,3\}$, the coefficients~\eqref{E: E def} satisfy~\eqref{E: exponent match}.
\end{theorem}

\begin{proof}
Proposition~\ref{prop:E-binomial} and Lemma~\ref{lem:T-binom-P} (the $n=1$ table of Proposition~\ref{prop:base-case} is the first case).
\end{proof}

\smallskip\noindent This closes the coefficient target~\eqref{E: exponent match}; with Propositions~\ref{prop:coeff-expansion} and~\ref{prop:duke-side} it completes the proof of Theorem~\ref{thm:main-finite}: $\mathcal{S}_{n}=\mathcal{D}_{n}=\beta'(-n)+(\log 4)\,\beta(-n)$ for every $n\ge 1$.

\section{Product side}\label{sec:typeN}

\noindent\textbf{Goal.} Develop the product side using the finite template~\eqref{E: S_n def}, block ratio~\eqref{E: block def}, and product~\eqref{E: Pn def} from Section~\ref{sec:main-results}. For odd~$n$, identify $\mathcal{L}_{n}=\log P_{n}$ with~$\mathcal{S}_{n}$.

\smallskip\noindent Recall $\mathcal{L}_{n}:=\log P_{n}$ when the product converges. Under absolute convergence the logarithm may be taken factorwise, giving $\mathcal{L}_{n}=\sum_{k\ge 0}\binom{n+k}{n}U_{n}(k)$: each block factor $B_{n,k}^{\binom{n+k}{n}}$ is positive and the series of absolute values converges (Lemma~\ref{lem:Pn-converge}). Theorem~\ref{thm:main-finite}, proved in Sections~\ref{sec:duke-side}--\ref{sec:closure}, identifies $\mathcal{S}_{n}=\mathcal{D}_{n}$ for all~$n$; Section~\ref{sec:Ln} shows $\mathcal{L}_{n}=\mathcal{S}_{n}$ when $n$ is odd.

\begin{lemma}[Odd-$n$ product convergence]\label{lem:Pn-converge}
If $n=2m-1$ is odd, the infinite product~$P_{n}$ in~\eqref{E: Pn def} converges absolutely and $\mathcal{L}_{n}=\log P_{n}=\sum_{k\ge 0}\binom{n+k}{n}U_{n}(k)$ is well defined, with $U_{n}(k)$ as in~\eqref{E: block def}.
\end{lemma}

\begin{proof}
For large~$k$, $B_{n,k}=1+O(k^{-1})$ so $U_{n}(k)=\log B_{n,k}=O(k^{-1})$. Expand $U_{n}(k)=\sum_{j\ge 1}c_{j}k^{-j}$; Lemmas~\ref{lem:low-power} and~\ref{lem:top-power-odd} give $c_{j}=0$ for $1\le j\le n+1$, so $U_{n}(k)=O(k^{-n-2})$. Since $\binom{n+k}{n}\sim k^{n}/n!$, the tail $\sum_{k}\binom{n+k}{n}|U_{n}(k)|$ converges. Finally, since $|B_{n,k}-1|=O(k^{-1})$, the factors are eventually positive, so absolute convergence justifies $\log P_{n}=\sum_{k\ge 0}\binom{n+k}{n}\log B_{n,k}$.
\end{proof}

\begin{remark}[Even-$n$ divergence, sketch]\label{rem:even-diverge}
If $n=2m\ge 2$ is even, Lemma~\ref{lem:low-power} cancels the coefficients of $k^{-1},\ldots,k^{-n}$ in the expansion of $U_{n}(k)=\log B_{n,k}$, so $U_{n}(k)=c_{n}k^{-n-1}+O(k^{-n-2})$ for a constant $c_{n}$. The sketch~\eqref{E: top power defect} identifies $c_{n}$ with a nonzero multiple of $\beta(-n)$ (equivalently $E_{n}/2\neq 0$). Since $\binom{n+k}{n}\sim k^{n}/n!$, the partial sums $\sum_{k\le K}\binom{n+k}{n}U_{n}(k)$ then grow like a nonzero multiple of $\log K$, so $P_{n}$ diverges. (This remark is not used by Theorems~\ref{thm:main-finite}--\ref{thm:main-odd}.)
\end{remark}

\begin{proposition}[Type-$N$ product, odd $n$]\label{prop:typeN-odd}
Let $n=2m-1$ be odd. With block sequences $(a_{k})$, $(b_{k})$ built from~$T(n,k)$ at quarter-integers $(4k+3)/4$ and $(4(n-k)+1)/4$, Lemmas~\ref{lem:low-power}--\ref{lem:top-power-odd} supply the power-sum hypotheses of Theorem~\ref{thm:typeN}(ii). Hence
\begin{align}\label{E: typeN odd}
\prod_{k=0}^{\infty}\left[ \frac{(4k+3)^{T(n,0)}\cdots(4k+4n+3)^{T(n,n)}}{(4k+1)^{T(n,n)}\cdots(4k+4n+1)^{T(n,0)}}\right]^{\binom{n+k}{n}}
=\prod_{k=0}^{n}\left( \frac{\Gamma_{n+1}\!\big(\frac{4(n-k)+1}{4}\big)}{\Gamma_{n+1}\!\big(\frac{4k+3}{4}\big)}\right)^{T(n,k)}.
\end{align}
\end{proposition}

\begin{proof}
Factor $4$ out of every slot: the block at index~$k$ is $\prod_{j}(k+u_{j})^{T(n,j)}/(k+v_{j})^{T(n,j)}$ with $u_{j}=(4j+3)/4$ and $v_{j}=(4(n-j)+1)/4$, since the powers of~$4$ cancel by Lemma~\ref{lem:tsum}. Verify $\sum a_{k}^{j}=\sum b_{k}^{j}$ for $j=1,\ldots,n+1$: for $j\le n$ use Lemmas~\ref{lem:low-power} and~\ref{lem:tsum}; at $j=n+1$ use Lemma~\ref{lem:top-power-odd}. Theorem~\ref{thm:typeN}(ii) then evaluates the product directly as
\[
\prod_{k=0}^{\infty}\Big[\cdots\Big]^{\binom{n+k}{n}}
=\prod_{k=0}^{n}\left(\frac{\Gamma_{n+1}(v_{k})}{\Gamma_{n+1}(u_{k})}\right)^{T(n,k)},
\]
which is~\eqref{E: typeN odd}.
\end{proof}

\section{Odd-$n$ product}\label{sec:Ln}

\noindent\textbf{Goal.} For odd~$n$, identify $\mathcal{L}_{n}=\log P_{n}$ with $\mathcal{S}_{n}$ and exponentiate to Theorem~\ref{thm:main-odd}.

\begin{proposition}[Odd-$n$ gamma normal form]\label{prop:Ln-odd}
If $n=2m-1$ is odd, then $\mathcal{L}_{n}=\mathcal{S}_{n}$, i.e.
\begin{align}\label{E: finite gamma normal odd}
\log P_{n}=\sum_{k=0}^{n}T(n,k)\left[\log\Gamma_{n+1}\!\Big(\frac{4(n-k)+1}{4}\Big)-\log\Gamma_{n+1}\!\Big(\frac{4k+3}{4}\Big)\right].
\end{align}
\end{proposition}

\begin{proof}
Take logarithms of~\eqref{E: typeN odd}; Lemma~\ref{lem:Pn-converge} identifies the left side with $\log P_{n}$.
\end{proof}

\smallskip\noindent Combining Proposition~\ref{prop:Ln-odd} with Theorem~\ref{thm:main-finite} and the classical vanishing $\beta(-n)=0$ for odd~$n$ (Lemma~\ref{lem:beta-neg}), we obtain Theorem~\ref{thm:main-odd}: $e^{\beta'(-n)}=P_{n}$.


\section{Outlook}\label{sec:outlook}

The beta identity provides a template: an $L$-function derivative at a negative integer equals a finite multiple-gamma template, which in turn matches an explicit block product with Eulerian weights. The FE closed-form match is tailored to quarter-integer arguments and Duke--Imamo\u{g}lu polynomials; it should extend to other settings where similar gamma expansions exist.

Natural follow-up questions include Dirichlet eta at negative integers, Hurwitz zeta derivatives at quarter-integers~\cite{article:15}, and analogous product formulas for other real Dirichlet $L$-functions. The design recipe of Section~\ref{sec:recipe} is the general programme: the beta theorem is its first complete instance for an $L$-function derivative, and the Wallis hierarchies and $\pi^{M!}$ products of Part~I mark out how much of the constant landscape the same mechanism already reaches.

\section*{Declarations}

\begin{itemize}
\item \textbf{Prior versions.} An early version of the Part~I material appeared as the author's preprint \emph{Generalising the Wallis Product} (arXiv:1906.00122, 2019); parts of this work were also submitted as an undergraduate essay at the University of Warwick and as a master's dissertation at Durham University. This manuscript substantially revises and extends that earlier material.
\end{itemize}

\bibliographystyle{amsplain}
\bibliography{beta-references}

\end{document}